\documentclass[10pt,a4paper]{amsart}
\usepackage{verbatim}

\usepackage[T1]{fontenc}
\usepackage{graphicx}
\usepackage{enumerate}
\usepackage{amsmath,amsfonts,amssymb}
\usepackage{color}
\usepackage{cite}
\definecolor{citation}{rgb}{0.2,0.58,0.2} 
\definecolor{formula}{rgb}{0.1,0.2,0.6}
\definecolor{url}{rgb}{0.3,0,0.5}
\usepackage{pgf,tikz}
\usepackage{mathrsfs}
\usepackage{fancyhdr}
\usepackage{dutchcal}

\usepackage{dsfont}

\newcommand{\reqnomode}{\tagsleft@false}

\usepackage[colorlinks,pdfpagelabels,pdfstartview = FitH,bookmarksopen = true,bookmarksnumbered = true,urlcolor=url,linkcolor = formula,plainpages = false,hypertexnames = false,citecolor = citation] {hyperref}

\def\dx{\,{\rm d}x}
\def\dy{\,{\rm d}y}
\def \d{\,{\rm d}}

\def\dist{\,{\rm dist}}

\allowdisplaybreaks
\makeatletter
\DeclareRobustCommand*{\bfseries}{%
  \not@math@alphabet\bfseries\mathbf
  \fontseries\bfdefault\selectfont
  \boldmath
}

\makeatother

\newlength{\defbaselineskip}
\newcommand{\setlinespacing}[1]
           {\setlength{\baselineskip}{#1 \defbaselineskip}}
\newcommand{\mint}{\mathop{\int\hskip -1,05em -\, \!\!\!}\nolimits}

\newtheorem{theorem}{Theorem}

\newtheorem{definition}{Definition}
\newtheorem{remark}{Remark}[section]
\newtheorem{lemma}{Lemma}[section]
\newtheorem{proposition}{Proposition}[section]
\numberwithin{equation}{section}
\allowdisplaybreaks[4]

\def\en{\mathbb N}
\def\er{\mathbb R}

\newcommand{\erN}{\mathbb{R}^N}
\newcommand\eps\varepsilon
\def\eqn#1$$#2$${\begin{equation}\label#1#2\end{equation}}

\newcommand{\ai}{a_{{\rm i}}}

\newcommand{\be}{\begin{equation}}
\newcommand{\oo}{\texttt{o}}
\newcommand{\ee}{\end{equation}}

\newcommand{\rr}{\varrho}

\newcommand{\snr}[1]{\lvert #1\rvert}

\newcommand{\nr}[1]{\lVert #1 \rVert}

\newcommand{\RN}{\mathbb{R}^{N}}

\def\name[#1, #2]{#1 #2}

\newcommand{\kk}{\kappa}

\newcommand{\rif}[1]{(\ref{#1})}
\newcommand{\trif}[1] {\textnormal{\rif{#1}}}

\newcommand{\stackleq}[1]{\stackrel{\rif{#1}}{ \leq}}

\newcommand{\cc}[1]{\textcolor[rgb]{0.00,0.00,0.00}{#1}}
\newcommand{\ccc}[1]{\textcolor[rgb]{0.00,0.00,0.00}{#1}}
\newcommand{\divo}{\textnormal{div}}
\def\loc{\operatorname{\rm{loc}}}

\title[
On the regularity of minima of non-autonomous functionals]{On the regularity of minima of  \\non-autonomous functionals
}
\author{Cristiana De Filippis}  \address{Cristiana De Filippis\\Mathematical Institute, University of Oxford\\ Andrew Wiles Building, Radcliffe Observatory Quarter, Woodstock Road, Oxford, OX26GG, Oxford, United Kingdom} \email{\texttt{Cristiana.DeFilippis@maths.ox.ac.uk}}
\author{Giuseppe Mingione}  \address{Giuseppe Mingione\\Dipartimento SMFI, Universit\`a di Parma, Viale delle Scienze 53/a, Campus, 43124 Parma, Italy} \email{\texttt{giuseppe.mingione@unipr.it}}
\begin{document}

\subjclass[2010]{35J60, 35J70\vspace{1mm}} 

\keywords{Regularity, non-autonomous functionals, $(p,q)$-growth\vspace{1mm}}

\thanks{{\it Acknowledgements.}\ This work is supported by the Engineering and Physical Sciences Research Council (EPSRC): CDT Grant Ref. EP/L015811/1. 
\vspace{1mm}}

\maketitle

\begin{abstract}
We consider regularity issues for minima of non-autonomous functionals in the Calculus of Variations exhibiting non-uniform ellipticity features. We provide a few sharp regularity results for local minimizers that also cover the case of functionals with nearly linear growth. The analysis is carried out provided certain necessary approximation-in-energy conditions are satisfied. These are related to the occurrence of the so-called Lavrentiev phenomenon 
that non-autonomous functionals might exhibit, and which is a natural obstruction to regularity. In the case of vector valued problems we concentrate on higher gradient integrability of minima. Instead, in the scalar case, we prove local Lipschitz estimates. We also present an approach via a variant of Moser's iteration technique that allows to reduce the analysis of several non-uniformly elliptic problems to that for uniformly elliptic ones. \end{abstract}
\vspace{3mm}
{\small \tableofcontents}

\setlinespacing{1.08}
\section{Introduction}
In this paper we collect a few results and techniques concerning the regularity of minima of non-autonomous elliptic functionals of the type
\eqn{genF}
$$
W^{1,1}(\Omega,\erN) \ni w   \mapsto\mathcal F(w, \Omega):= \int_{\Omega} F(x, Dw)  \dx\;.
$$
In \rif{genF}, as in the rest of the paper, $\Omega \subset \er^n$ denotes a bounded open domain, for $n\geq 2$. The function $F \colon \Omega \times \er^{N\times n}\to [0, \infty)$ is Carath\'eodory regular and $N\geq 1$; we also assume that, whenever they are considered, derivatives of $F(\cdot)$ with respect to the gradient variable are also Carath\'eodory regular. The case $N>1$ is usually appealed 
to as the vectorial case. In our setting a function $u \in W^{1,1}_{\loc}(\Omega,\er^N)$ is a \emph{local minimizer} of the functional $\mathcal F$ in \eqref{genF} if $F(\cdot, Du)\in L^1_{\loc}(\Omega)$ and $\mathcal F(u;\tilde \Omega)\leq \mathcal F(w;\tilde \Omega)$ holds for every competitor $w \in u + W^{1,1}_0(\tilde \Omega; \er^N)$ and for every open subset $\tilde \Omega\Subset \Omega$. 

The main point here is that the functionals in question here exhibit non-uniform ellipticity features. These emerge when looking at the Euler-Lagrange equation 
$
\divo\, \partial_{z} F(x,Du)=0,
$
whose rate of non-uniform ellipticity is quantified by the ratio \cc{$\mathcal R(z, B)$} (on any ball $B\subset \Omega$) 
$$
\mathcal R(z, B):= \frac{\mbox{$\sup_{x\in B}$ of the highest eigenvalue of}\ \partial_{zz} F(x,z)}{\mbox{$\inf_{x\in B}$ of the lowest eigenvalue of}\  \partial_{zz} F(x,z)} $$
that in the non-uniformly elliptic case becomes in fact unbounded as $|z|\to \infty$. For instance, this is not the case of $p$-Laplacean type functionals, i.e., $F(x, z)\approx |z|^p$, for which $\mathcal R(z, B)\equiv 1$. See \cite{KM1, KM2, Dark, Uh, Ur} for regularity results in this situation. This is instead the case of the double phase functional \cite{BCM3, CM2, Z2, Z3}
\eqn{duefasi}
$$
w
\mapsto \int_{\Omega} \left(|Dw|^p + a(x)|Dw|^q\right)  \dx\;, \qquad 0\leq a(\cdot)\in L^{\infty} \qquad 1<p<q\;,
$$
where it is $\mathcal R(z, B)\approx  \ccc{1+\|a\|_{L^{\infty}(B)}|z|^{q-p}}$ on any ball $B$ intersecting $\{a(x)=0\}$. Another instance is given by the variable exponent energy 
\eqn{pLapx}
$$
 w   \mapsto  \int_{\Omega} |Dw|^{p(x)}   \dx\;, \qquad p(x)>1
$$
and in this case it is $\mathcal R(z, B)\approx  |z|^{p_+-p_-}$
, for $|z|$ large, where $p_-:= \min_B \, p(x)$ and $p_+:= \max_B \, p(x)$. There is a by now extensive literature on the regularity for minima of functionals \rif{duefasi}-\rif{pLapx}, see for instance \cite{BCM1, BCM3, byun1, byun2, CM2, CM3, me, demicz, demicz2, deoh, depa, PRR, Radu1, RT1} and \cite{Dark, Raduover} for overviews.  More in general, larger classes of functionals defined in so-called Musielak-Orlicz spaces are defined by
\eqn{PPhi}
$$
 w   \mapsto  \int_{\Omega} \Phi(x, |Dw|)   \dx\;, 
$$
where, \cc{$\Phi \colon \Omega \times [0, \infty)\to [0,\infty)$} is a Carath\'eodory function such that for each choice of $x \in \Omega$, the partial map $t \mapsto \Phi(x, t)$ is a Young function and thereby generates an Orlicz space (that changes with $x$). For this we refer to  \cite{HHT, hastobook, Hasto, Hastok, Radu2}. The common feature of many of such functionals is that they satisfy the so-called $(p,q)$-growth conditions
\eqn{cond-p}
$$
|z|^p \lesssim F(x, w, z)\lesssim |z|^q+1\,, \qquad \mbox{for}\ 1 < p  < q\;.
$$
We refer to the basic papers of Marcellini \cite{M1, M2, M3}, where the first regularity results have been obtained under assumptions \rif{cond-p}.

In this paper we want to collect a few general results on functionals of the type \rif{genF} under $(p,q)$-growth conditions as in \rif{cond-p}, that extend those in available literature, and in various directions. For instance, we consider conditions where the only possible polynomial bound from below as in \rif{cond-p} is $p=1$. Specifically, we relax the lower bound in \rif{cond-p} to allow nearly linear growth conditions  in the gradient; in this case a model is
\eqn{switch}
$$w \mapsto \int_{\Omega} \left[|Dw|\log(1+|Dw|)+a(x)(1+|Dw|^2)^{\frac q2}\right]  \dx\,, \quad 0\leq a(\cdot)\in L^{\infty}\,,\ \quad 1<q\;.$$
Further examples are in Remark \ref{estensione} below. 
Some of our a priori estimates techniques can also be used in different, more geometric settings. In this case a relevant model functional is
 \eqn{switchlinear}
$$w \mapsto \int_{\Omega} \left[(1+|Dw|^\kappa)^{1/\kappa}+a(x)(1+|Dw|^2)^{\frac q2}\right]  \dx\;, \qquad 0\leq a(\cdot)\in L^{\infty} \qquad 1<q,\kappa\;.$$
This has linear growth in the gradient on the set $\{a(x)=0\}$. The full treatment of functionals as in \rif{switchlinear} involves a suitable use of relaxed functionals and spaces of BV functions \cite{Bi, BS1, franz1, franz2}. An instance of the results included here is
\begin{theorem}\label{sample1} Let $u \in W^{1,1}_{\loc}(\Omega)$ be a local minimizer of the functional in 
\trif{switch} and assume
\begin{flalign}\label{pqsample}
0 \leq a(\cdot) \in W^{1,r}_{\loc}(\Omega) \cap L^{\infty}_{\loc}(\Omega)\qquad \mbox{and}\qquad 1< q<1+\frac{r-n}{nr}\;.
\end{flalign}
Then $u$ il locally Lipschitz regular in $\Omega$. 
\end{theorem}
The result for the autonomous case $a(\cdot)=0$
has been established in \cite{MS}; see also \cite{Bi, FM, M3}. Theorem \ref{sample1} is a particular case of Theorem \ref{main3} below and Section  \ref{primot} for the proof of Theorem \ref{sample1}. 
Let us explain why assumptions \rif{pqsample} are in a sense sharp. The functional \rif{switch} can be seen as the limit case of the one in \rif{duefasi} when $p\to 1$. For the functional in \rif{duefasi} the local Lipschitz continuity of minima is guaranteed by the assumption  
\begin{flalign}\label{pqsample2}
0 \leq a(\cdot) \in C^{0, \alpha}_{\loc}(\Omega)\qquad \mbox{and}\qquad q\leq p+\frac{p\alpha}{n}\;,
\end{flalign}
which is optimal by \cite{ELM2, FMM}. See \cite{BCM3, CM2} for regularity results, instead. Sobolev embedding gives that $ a(\cdot) \in W^{1,r}_{\loc}$ implies $a(\cdot) \in C^{0, \alpha}_{\loc}$, where $\alpha = 1-n/r$. In turn, substituting this value of $\alpha$ in \rif{pqsample2} and taking $p=1$, makes \rif{pqsample} and \rif{pqsample2} coincide (apart from the equality case in \rif{pqsample2}, due to the peculiar structure in \rif{duefasi}). Assumption \rif{pqsample} describes the catch between $p,q$ and the H\"older continuity exponent $\alpha$ as in \rif{pqsample2}, but in a weakly differentiable version. This approach has been introduced in the interesting papers \cite{EMM, EMM2}, where Moser's iteration has been employed;  previous results involving Sobolev coefficients appear in \cite{KM3}. One our goals here is to describe a variant of Moser's iteration, that, in a sense, allows to treat non-uniformly elliptic equations as uniformly elliptic ones. See Section \ref{moser2} below.

A second result of this paper deals with the higher integrability of minima in the general vectorial case, and avoids considering differentiability assumptions on coefficients. For general non-autonomous convex functionals \rif{genF} with $(p,q)$-growth as in \rif{cond-p}, the assumption of (uniform) $\alpha$-H\"older continuity of the partial map
\eqn{holdy}
$$
x \mapsto \frac{\partial_z F(x,z)}{1 +|z|^{q-1}}
$$
guarantees that any local minimizer, which is by \rif{cond-p} only in $W^{1,p}_{\loc}$, actually belongs to the  smaller space $W^{1,q}_{\loc}$, provided $q/p<1+\alpha/n$ and the Lavrentiev phenomenon does not appear \cite{sharp} (see  \rif{lav1} below). When applied to the functional in  \rif{duefasi}, condition \rif{holdy} amounts to require that \rif{pqsample2} is satisfied. On the other hand, as seen in \cite{BCM3, CM2, CM3} for the specific functional in \rif{pqsample2}, considering bounded minimizers allows to improve the bound in \rif{pqsample2}. More precisely, condition \rif{pqsample2} can be replaced by 
\begin{flalign}\label{pqsample3}
0 \leq a(\cdot) \in C^{0, \alpha}_{\loc}(\Omega)\,,  u \in L^{\infty}_{\loc}(\Omega) \qquad \mbox{and}\qquad q\leq p+\alpha\;.
\end{flalign}
This is again sharp \cite{sharp, FMM}. 
Here we see that conditions as in \rif{pqsample3} actually work for general functionals as in \rif{genF} and imply higher gradient integrability of minima; see Theorem \ref{main1} below. For autonomous functionals $w   \mapsto \int_{\Omega} F(Dw)  \dx$ the interaction between boundedness of minima and dimensionless bounds has been considered in \cite{CKP, choe, ELM0}. 

\subsection{The Lavrentiev gap} In this section, unless otherwise specified, we deal with a functional as in \rif{genF}, where $z\mapsto  F(x, z)$ is convex (for a.e.~$x \in \Omega$) and with the additional lower bound $\bar{F}(\snr{z})\lesssim F(x,z)$, where $ \bar{F}\colon [0, \infty)\to [0, \infty)$ has superlinear growth in the sense of \rif{assF3}$_2$ below. In this situation the so-called Lavrentiev phenomenon might appear. For instance, under $(p,q)$-growth conditions \rif{cond-p}, there might occur an inequality of the type 
\begin{equation}\label{Lav2}
\inf_{w\in u_0+W^{1,p}_0(\Omega,\er^N)} \int_{\Omega} F(x,Dw) \ dx < \inf_{w\in u_0+W^{1,q}_0(\Omega,\er^N)}
\int_{\Omega} F(x,Dw) \ dx\;,
\end{equation}
for a suitable (even smooth) boundary datum  $u_0$. 
In other words, it is not possible to achieve the minimum of the functional via more regular maps, although these are dense. This is a tautological obstruction to regularity of minima, and indeed several counterexamples in regularity are based on the occurrence of \rif{Lav2} \cite{sharp, FMM, Z2, Z3}. In this paper we further develop the approach of \cite{sharp}, proving regularity via a suitable analysis of Lavrentiev phenomenon. This goes as follows. First observe that the convexity of $z \to F(\cdot, z)$ guarantees lower semicontinuity, in the sense that 
\eqn{lower}
$$
\mathcal F (u, \Omega) \leq \liminf_{j} \mathcal F(u_j, \Omega)
$$
holds for all $\{u_j\}\subset W^{1,1}(\Omega, \er^N)$ such that $u_j \rightharpoonup u$ weakly in  
$W^{1,1}(\Omega,\Bbb R^{N})$. As in \cite{ABF} for $q\geq 1$, we define, whenever $B \Subset \Omega$ is a ball, the relaxed functional 
$$
\mathcal{F}^{q}(u,B)  \, := \inf_{\{u_j\}\subset W^{1,q}(B, \er^N)}  \left\{ \liminf_j  \mathcal{F}(u_j,B)  \, \colon \,  u_j \to u\  \mbox{in} \ L^{1}(B,\er^N)  \right\} 
$$
for every $u\in W^{1,1}(B, \er^N)$. Accordingly, as in \cite{sharp} we consider the Lavrentiev gap
\eqn{lav1}
$$
\mathcal {L}^{q}(u,B):= \mathcal{F}^{q}(u,B)- \mathcal{F}(u,B)\;.
$$
We refer to \cite{ABF} for a related and extended definition, allowing to show that, in certain cases, $\mathcal {L}^{q}(u,B)$ is a measure which is singular with respect to the Lebesgue measure. By \rif{lower} and \rif{assF3}$_2$ it is $\mathcal{F}^{1}\equiv \mathcal{F}$ and moreover $1\leq q_1\leq q_2$ implies $\mathcal{F}^{q_1}\leq \mathcal{F}^{q_2}$. In the case a lower bound $|z|^p\lesssim F(x, z)$ for $p\geq 1$ is satisfied, it holds that 
$\mathcal{F}^{p}\equiv \mathcal{F}$. 
Examples for which $\mathcal {L}^{q}(\cdot,B)\not=0$ occur \cite{ABF, sharp, Z2, Z3}, and this in fact relates to \rif{Lav2} and to the approximation in energy in the following sense:
\begin{proposition} \label{approssi1}
 Let $u\in W^{1,1}(B,\Bbb R^{N})$ be a function such that
$\mathcal{F}(u,B)<\infty$, where $B \subset \Omega$ is a fixed ball. Then $\mathcal {L}^q(u,B)=0$
 iff there
exists a sequence $\{u_{j}\} \subset
W^{1,q}(B,\Bbb R^{N})$
such that $u_{j} \rightharpoonup u$ weakly in  
$W^{1,1}(B,\Bbb R^{N})$ and $\mathcal{F}(u_{j},B) \to
\mathcal{F}(u,B).$
\end{proposition}
The proof is a straightforward consequence of the definitions and of the fact that the lower bound $\bar{F}(|z|)\lesssim F(x,z)$ allows to consider weakly convergent sequences via Dunford-Pettis criterion. In this paper we prove that regularity of local minimizers $u$ holds provided a suitable Lavrentiev gap vanishes on $u$, a condition, that, in a sense, is tautologically necessary for regularity. A main point here is that, in fact, in several examples, the assumptions guaranteeing that the Lavrentiev gap vanishes 
are the same allowing for a priori estimates, thereby closing the circle. See also Section \ref{absence} below. Notice that this is the case when no $x$-dependence is allowed: plain convexity of $z \mapsto F(z)$ suffices. A most interesting example is given by the double phase functional \rif{duefasi}, where conditions for regularity \rif{pqsample2} allow to prove that the gap vanishes \cite{sharp}. As anticipated in the previous section, there is an interplay between bounds on the gap $q/p$ and a priori boundedness of minima. An instance is given by the following fact from \cite{BCM3}:
\begin{theorem}\label{bcmtheorem} Let $ u \in W^{1,1}_{\loc}(\Omega,\er^N)$ be a local minimizer of 
the functional $\mathcal F$ in \trif{genF} satisfying 
\eqn{doppioc}
$$
|z|^p + a(x)|z|^q  \lesssim F(x,z) \lesssim |z|^p + a(x)|z|^q +1\;,
$$
with \trif{pqsample2} being in force. 
Then, for every ball
 $B\Subset  \Omega$ there exists a sequence $\{u_j\} $ of $W^{1,\infty}(B,\er^N)$-regular functions such that $u_j \to u$ strongly in $W^{1,p}(B,\er^N)$ and in $L^{\infty}(B,\er^N)$,
and  $
 \mathcal F (u_j, B) \to \mathcal F (u, B). 
 $
\end{theorem}
Notably, in Theorem \ref{bcmtheorem} no convexity of $z \mapsto F(\cdot, z)$ is assumed, i.e., the double-sided control in \rif{doppioc} suffices. 
Theorem \ref{bcmtheorem} leads to define a different relaxation of the functional in \rif{genF}; specifically, we have for every $u \in W^{1,1}(B, \er^n)\cap L^{\infty}(B,\er^N)$
 \label{rilassato1}
$$
\mathcal{F}_{b}^{q}(u,B)  \, := \inf_{\{u_j\}\subset W^{1,q}(B, \er^n)\cap L^{\infty}(B,\er^N)} \left\{ \liminf_j \mathcal{F}(u_j,B)  \, \colon \,  u_j \to u\  \mbox{in} \ L^{\infty}(B,\er^N)  \right\}
$$
and, finally
\eqn{lav2}
$$
\mathcal {L}_b^{q}(u,B):= \mathcal{F}_{b}^{q}(u,B)- \mathcal{F}(u,B)\;.
$$
Similarly to Proposition \ref{approssi1}, we have 
\begin{proposition} \label{approssi2}
 Let $u\in W^{1,1}(B,\Bbb R^{N})\cap L^{\infty}(B, \er^N)$ be a function such that
$\mathcal{F}(u,B)<\infty$, where $B \subset \Omega$ is a fixed ball. Then $\mathcal {L}_b^{q}(u,B)=0$
 iff there
exists a sequence $\{u_{j}\} \subset
W^{1,q}(B,\Bbb R^{N})\cap L^{\infty}(B,\er^N)$
such that $u_{j} \rightharpoonup u$ weakly in  
$W^{1,1}(B,\Bbb R^{N})$, $\|u_j-u\|_{L^{\infty}(B,\er^N)}\to 0$ and $\mathcal{F}(u_{j},B) \to
\mathcal{F}(u,B).$
\end{proposition}
\subsection{Regularity via Lavrentiev gap}\label{lav-sec}
We consider in integrand $F \colon \Omega \times \er^{N\times n}\to [0, \infty)$ such that $z \to F(\cdot,z)$ is locally $C^{1}$-regular and satisfies 
\eqn{assF}
$$
\left\{
\begin{array}{c}
\nu |z|^{p}\leq F(x,z) \leq L(1+ |z|^{q})\\ [6 pt]
\nu\left(\lambda^{2}+ |z_{1}|^{2} +
  |z_{2}|^{2}\right)^{\frac{p-2}{2}} |z_{1}-z_{2}|^{2} \leq
\left( \partial_z F(x,z_{1})-\partial_z F(x,z_{2})\right)\cdot\left(z_{1}-z_{2}\right) \\ [6 pt]
|\partial_z F(x,z)-\partial_z F(y,z)|
\leq L |x-y|^{\alpha}(1+|z|^{q-1})\;,
 \end{array}\right.
 $$
whenever $x,y\in \Omega$, $z, z_1, z_2 \in \er^{N\times n}$, where $1<p\leq q $, $\lambda \in [0,1]$, $\alpha \in (0,1]$ and $0< \nu \leq 1 \leq L$ are fixed constants. Notice that $F(\cdot)$ is not assumed to twice differentiable here with respect to the gradient variable; in particular, no growth assumption on second derivatives of $F(\cdot)$ is considered here. The monotonicity inequality in \rif{assF}$_2$ implies that $z \to F(\cdot,z)$ is convex. In turn, this and \rif{assF}$_1$ imply that 
\eqn{crescitader}
$$
 \left| \partial_z F(x,z)\right| \leq c(1+ |z|^{2})^{\frac{q-1}{2}}
$$
holds too, for every $z \in \er^{N\times n}$ and $x\in \Omega$, where $c\equiv c (L,q)$. 
\begin{theorem}\label{main1}
Let $u\in W^{1,1}_{\loc}(\Omega,\RN)\cap L^{\infty}_{\loc}(\Omega,\RN)$ be a local minimizer of the functional $\mathcal F$ in \eqref{genF} under assumptions \trif{assF}, with 
\eqn{mainbound}
$$1  <p < q < p+\alpha\min\left\{1, \frac{p}{2}\right\}\;.$$ 
Assume that 
\eqn{zerogap}
$$\mathcal{L}_{b}^{q}(u,B_R)=0$$ holds for a ball $B_R\Subset \Omega$ with $R\leq 1$. If $\tilde{p}$ is such that
\eqn{ptilde}
$$
q< \tilde{p}<p+\alpha\min\left\{1, \frac{p}{2}\right\}
$$
and $B_{\rr}\Subset B_{R}$ is ball concentric to $B_R$, then
\eqn{main-ine}
$$
\nr{Du}_{L^{\tilde{p}}(B_{\rr})}\le \frac{c}{(R-\rr)^{\kappa_{1}}}\left[1+\mathcal{F}(u,B_{R})\right]^{\kappa_{2}}
$$
holds for a constant $c$ depending on $n,N,\nu,L,p,q,\alpha,\tilde p, \nr{u}_{L^{\infty}(B_{R})}$, and exponents $\kappa_{1}, \kappa_{2}\equiv \kappa_{1}, \kappa_{2}(n,p,q,\alpha, \tilde p)$. In particular, if \trif{zerogap} holds for every such ball $B_R \Subset \Omega$, then $u\in W^{1,q}_{\loc}(\Omega,\RN)$. 
\end{theorem}
It remains to establish when \rif{zerogap} is satisfied. This is discussed Ssection \ref{absence} below. 
\begin{remark}\label{sharpdomina}
\emph{ The result of Theorem \ref{main1} is new only for $p< n$. Indeed, $p\geq n$ implies $p+\alpha \leq  p+ p\alpha/n$ and the assertion of Theorem \ref{main1} is implied by the one in \cite{sharp}, that works  assuming the bound $q/p< 1+\alpha/n$. On the other hand, for $p>n$ minimizers are automatically bounded, and the main assumption in Theorem \ref{main1}, i.e., $u \in L^{\infty}_{\loc}(\Omega,\RN)$, looses its meaning.}
\end{remark}

\subsection{Conditions implying absence of the gap}\label{absence} A first class of integrands for which $\mathcal{L}^{q}=\mathcal{L}_{b}^{q}=0$
holds is given by those satisfying a double-sided control of the type 
\eqn{db1}
$$
a_0(x)F_0(z) \lesssim F(x,z)\lesssim  a_0(x)F_0(z)  +1\;. 
$$
Here $0 < \nu \leq a_0(x)\leq L $ is a measurable function and $F_0(\cdot)$ is non-negative and convex; see for instance \cite{EMM, sharp}. 
To extend \rif{db1}, one can consider the setting of so-called Musielak-Orlicz spaces, widely discussed in \cite{hastobook}. In this case we replace \rif{db1} by the more general 
\eqn{lavvi}
$$
\Phi(x, |z|) \lesssim F(x,z)\lesssim  \Phi(x, |z|) +1\;,
$$ 
where $\Phi \colon \Omega \times [0, \infty) \to [0, \infty)$ is a Carath\'eodory function which is convex in the second variable; the relation with functionals as in \rif{PPhi} is obvious. Examples are again given by the variable exponent energy
$\Phi(x, |z|) \equiv  |z|^{p(x)}$ and of course by $\Phi(x, |z|) \equiv   |z|^{p}+a(x)|z|^q$; see Theorem \ref{bcmtheorem}. In the setting of \rif{lavvi} the absence of Lavrentiev phenomenon is strongly related to the density of smooth functions and the boundeness of maximal operators in related Musielak-Orlicz dspaces. In general these assumptions are again closely tied to those guaranteeing regularity of minima of corresponding functionals \rif{PPhi}. For such issues we refer to \cite{the4, HHT, hastobook, Hasto}. A general setting is described in \cite{sharp}. Further results in this direction can be found in \cite{CM2, CM3}, and we refer also to \cite{Dark} for a general overview. 
\subsection{Lipschitz estimates} We now consider the issue of Lipschitz regularity of minima of functionals as in \rif{genF}. This does not hold in the general vectorial case, and we therefore concentrate on the scalar one $N=1$. Several of the arguments developed here can be anyway adapted to the vectorial case as well, provided suitable structure conditions are assumed, i.e., $F(x, z)\equiv F(x, |z|)$ (see for instance \cite{BeckM, EMM}). We are not going to pursue this path here. The assumptions on the integrand $F(\cdot)$ in \trif{genF} are now as follows:
\eqn{assF2}
$$
\left\{
\begin{array}{c}
z\mapsto F(\cdot,z)\in C^{2}_{\loc}(\er^n\setminus \{0\})\cap C^{1}_{\loc}(\er^n)\\ [6 pt]
x\mapsto \partial_zF(x,z)\in W^{1,r}(\Omega,\er^n)\quad \mbox{for every $z \in \er^n$}\\ [6 pt]
\nu \bar{F}(\snr{z})+ \nu(\lambda^{2}+\snr{z}^{2})^{\frac{2-\mu}{2}}\le F(x,z)\le L(\lambda^{2}+\snr{z}^{2})^{\frac{q}{2}}+L (\lambda^{2}+\snr{z}^{2})^{\frac{2-\mu}{2}}\\ [6 pt]
\nu (\lambda^{2}+\snr{z}^{2})^{-\frac{\mu}{2}}\snr{\xi}^{2}\le \partial_{zz}F(x,z)\, \xi\cdot\xi\\ [6 pt]
\snr{\partial_{zz}F(x,z)}\le L(\lambda^{2}+\snr{z}^{2})^{\frac{q-2}{2}}+  L(\lambda^{2}+\snr{z}^{2})^{-\frac{\mu}{2}}\\ [6 pt]
\ \snr{\partial_{xz} F(x,z)}\le Lh(x)(\lambda^{2}+\snr{z}^{2})^{\frac{q-1}{2}}+ Lh(x) (\lambda^{2}+\snr{z}^{2})^{\frac{1-\mu}{2}} \;.
 \end{array}\right.
 $$
Conditions \rif{assF2} are assumed to hold for every choice of $z,\xi \in \er^n$, $|z|\not=0$, and for a.e. $x \in \Omega$, where 
$\lambda \in [0,1]$, and $0< \nu \leq 1 \leq L$ are fixed constants.  We initially require that $q\geq 2-\mu$, $\mu<2$ and $r>n$. The two functions $h \colon \Omega \to [0, \infty)$ and $ \bar{F}\colon [0, \infty)\to [0, \infty)$ satisfy
\eqn{assF3}
$$
\left\{
\begin{array}{c}
h(\cdot)\in L^{r}(\Omega) \ \ \mbox{(recall $r>n$)} \\ [8 pt]
\displaystyle
\lim_{t\to \infty}\frac{\bar{F}(t)}{ t}=\infty
\;.
\end{array}
\right.
$$
\begin{theorem}\label{main3}
Let $u\in W^{1,1}_{\loc}(\Omega)$ be a local minimizer of the functional $\mathcal F$ in \eqref{genF} under assumptions \trif{assF2}-\trif{assF3}, with 
\begin{flalign}\label{pq}
\mu <2\,, \qquad 1<q\,,\qquad  1\le \frac{q}{2-\mu}<1+\frac{r-n}{nr}\;.
\end{flalign}
Assume that 
\eqn{zerogap2}
$$\mathcal{L}^{q}(u,B_R)=0$$ holds for a ball $B_R\Subset \Omega$ with $R\leq 1$. If $B_{\rr}\Subset B_{R}$ is another ball concentric to $B_R$, then
\eqn{main-ine-lip}
$$
\nr{Du}_{L^{\infty}(B_{\rr})}\le c\left(\frac{L+L\|h\|_{L^{r}(B_R)}}{R-\rr}\right)^{\kappa_1}\left[1+\mathcal{F}(u,B_{R})\right]^{\kappa_{2}}
$$
holds for $c\equiv c(n,\nu, \mu,q,r)$ and $\kappa_{1}, \kappa_{2}\equiv \kappa_1, \kappa_2(n,\mu,q,r)$. In particular, if \trif{zerogap2} holds for every such ball $B_R \Subset \Omega$, then $u\in W^{1,\infty}_{\loc}(\Omega)$. \end{theorem}
\begin{remark}\label{r6} 
\emph{The condition $R\leq1$ in Theorem \ref{main3} can obviously be dropped; we assumed it to make the proof more transparent. As all our results are local, we can put  $W^{1,r}_{\loc}$ and $L^r_{\loc}$ in \eqref{assF2}$_2$ and \eqref{assF3}$_1$, respectively. The $(p,q)$-growth setting can be recovered with the choice $\mu = 2-p$ and $\bar{F}(t)=(\lambda^2+t^2)^{p/2}$; see Section \ref{moser2} below. The exponents $\kk_1, \kk_2$ in \rif{main-ine-lip} can be explicitly computed (see Remark \ref{esp1} below) and they do coincide with those of the standard $(p,q)$-case when focusing on this situation (see Remark \ref{esp2} below). The one \rif{assF2}$_4$ is known as $\mu$-ellipticity condition and it is of common use in problems with linear and nearly-linear growth \cite{BS1,Bi,franz1}. It has been introduced in \cite{FM}.}\end{remark}
\begin{remark}\label{estensione} 
\emph{The technique considered here can be modified using a by now standard truncation argument in the gradient, as for instance in \cite{BeckM, EMM}. In this way we can also prescribe that assumptions \rif{assF2}$_{4,5,6}$ are satisfied only for $|z|>T$, for a fixed non-negative number $T$, but still considering convex integrands $F(\cdot)$. This is not surprising, as in order to get local Lipschitz regularity of minima only the behaviour of the functional for large values of the gradient matters. This allows for instance to treat functionals of the type
$$
w \mapsto \int_{\Omega} \left[\bar{F}(|Dw|)+ a(x)\bar{F}_q(|Dw|)\right] \dx 
$$ for convex integrands $\bar{F}(\cdot)$ and $\bar{F}_q(\cdot)$ such that $\bar{F}_q(t)\approx t^q$ for $t$ large and \rif{assF3}$_2$ holds. An instance is given by
$$
w \mapsto \int_{\Omega} \left[ \bar{F}_k(|Dw|) + a(x)(\lambda^2+|Dw|^2)^{\frac q2}\right] \dx , 
$$
where $\lambda \in [0,1]$ and 
$$
F_k(t) \approx t L_k(1+t) \,,\qquad  \left\{ \begin{array}{c} 
L_1(t):= \log (1+t)\\ [4 pt]
L_{k+1}(t):= \log \left(1+L_k(t)\right)
\end{array}\,, \right. \qquad k \in \en\;.
$$
When $a(x)\equiv 0$ such functionals are considered in \cite{FM, M3}.}\end{remark}
\begin{remark}\label{ontheass} \emph {In Theorem \ref{main3} we can assume 
$
F(x, 0)=0.
$ 
This can be seen by replacing $F(x, z)$ with $F(x,z)-F(x,0)$. Next, the standard proof of Morrey's embedding theorem gives that
$$
|\partial_z F(x_1, z)-\partial_z F(x_2, z)| \leq c L\|h\|_{L^r(\Omega)}\left[(\lambda^2+|z|^2)^{\frac{q-1}{2}}+(\lambda^2+|z|^2)^{\frac{1-\mu}{2}}\right] |x_1-x_2|^{1-\frac{n}{r}}
$$
holds for $c\equiv c(n,q,r)$ whenever $x_1, x_2\in \Omega$ and $z\in \er^n$. Integrating this last inequality and using
$F(x_1,0)=F(x_2,0)=0$, we conclude with
\eqn{holderF}
$$
|F(x_1, z)-F(x_2, z)| \leq cL\|h\|_{L^r(\Omega)}\left[(\lambda^2+|z|^2)^{\frac{q}{2}}+(\lambda^2+|z|^2)^{\frac{2-\mu}{2}}\right] |x_1-x_2|^{1-\frac{n}{r}}
$$
again  for $c\equiv c(n,q,r)$.
}
\end{remark}
\section{Preliminaries}
In this paper we denote by $c$ a general constant larger than one. Different occurences from line to line will be still denoted by $c$, while special occurrences will be denoted by $c_1, c_2,  \tilde c$ and so on. Relevant
dependencies on parameters will be emphasised using parentheses, i.e., $c_{1}\equiv c_1(n,p)$ means that $c_1$ depends on $n,p$. In a similar fashion, by $\oo(\kappa)$ we denote a quantity depending on the parameter $\kappa$ such that $\oo(\kappa)\to 0$ when $\kappa$ goes to a relevant limit (typically $\kappa \to 0$ or $\kappa \to \infty$); also in this case the expression of $\oo(\kappa)$ might vary from line to line and relevant dependences are emphasized. We denote by $ B_r(x_0):=\{x \in \er^n \, : \,  |x-x_0|< r\}$ the open ball with center $x_0$ and radius $r>0$; when no ambiguity arises, we omit denoting the center as follows: $B_r \equiv B_r(x_0)$. Very often, when not otherwise stated, different balls in the same context will share the same center. When considering function spaces of vector valued maps, such as $L^p(\Omega,\er^k)$, $W^{1,p}(\Omega,\er^k)$ etc, we often abbreviate as $L^p(\Omega)$, $W^{1,p}(\Omega)$ and so on; the meaning will be clear from the context. With $\mathcal B \subset \er^{n}$ being a measurable subset with finite and positive measure $|\mathcal B|>0$, and with $g \colon \mathcal B \to \er^{k}$, $k\geq 1$, being a measurable map, we denote by  $$
   (g)_{\mathcal B} \equiv \mint_{\mathcal B}  g(x)  \dx  := \frac{1}{|\mathcal B|}\int_{\mathcal B}  g(x)  \dx
$$
its integral average. We now recall a few basic facts concerning fractional Sobolev spaces. 
\begin{definition}\label{fra1def}
Let $\alpha \in (0,1)$, $p \in [1, \infty)$, $k \in \en$, and let $\Omega \subset \er^n$ be an open subset with $n\geq 2$ (we allow for the case $\Omega =\er^n$). The fractional Sobolev space $W^{\alpha ,p}(\Omega,\er^k )$ is defined prescribing that $f \colon \Omega \to \er^k$ belongs to  $W^{\alpha ,p}(\Omega,\er^k )\equiv W^{\alpha ,p}(\Omega)$ iff
the following Gagliardo type norm is finite:
\begin{eqnarray*}\| f \|_{W^{\alpha ,p}(\Omega )} & := &\|f\|_{L^p(\Omega,\er^k)}+ \left(\int_{\Omega} \int_{\Omega}  
\frac{|f(x)
- f(y) |^{p}}{|x-y|^{n+\alpha p}} \ dx \, dy \right)^{1/p}\\
&=:& \|f\|_{L^p(\Omega,\er^k)} + [f]_{\alpha, p;\Omega}\,.
\end{eqnarray*}
Accordingly, in the case $\alpha = [\alpha]+\{\alpha\}\in \en + (0,1)>1$, we say that $f\in W^{\alpha ,p}(\Omega,\er^k )$ iff the following quantity is finite
\eqn{derfrac}
$$
\| f \|_{W^{\alpha ,p}(\Omega )}  := \| f \|_{W^{[\alpha],p}(\Omega )} +[D^{[\alpha]}f]_{\{\alpha\}, p;\Omega}\;.
$$
The local variant $W^{\alpha ,p}_{\loc}(\Omega,\er^k )$ is defined by requiring that $f \in W^{\alpha ,p}_{\loc}(\Omega,\er^k )$ iff $f \in W^{\alpha ,p}(\tilde{\Omega},\er^k)$ for every open subset $\tilde{\Omega} \Subset \Omega$. 
\end{definition}
For a map $f \colon \Omega \to \er^k$ and a vector $h \in \er^n$, we denote by $\tau_{h}\colon L^1(\Omega,\er^k) \to L^{1}(\Omega_{|h|},\er^k)$ the standard finite difference
operator pointwise defined as
$
\tau_{h}f(x)\equiv \tau_{h}f(x)\equiv \tau_{h}(f)(x):=f(x+h)-f(x)
$, 
whenever $\Omega_{|h|}:=\{x \in \Omega \, : \, 
\dist(x, \partial \Omega) > |h|\}$ is not empty.
\begin{definition}\label{fra2def}
Let $\alpha \in (0,1)$, $p \in [1, \infty)$, $k \in \en$, and let $\Omega \subset \er^n$ be an open subset with $n\geq 2$. The Nikol'skii space $N^{\alpha,p}(\Omega,\er^k )$ is defined prescribing that $f \in N^{\alpha,p}(\Omega,\er^k )$ iff 
$$\| f \|_{N^{\alpha,p}(\Omega,\er^k )} :=\|f\|_{L^p(\Omega,\er^k)} + \left(\sup_{|h|\not=0}\, \int_{\Omega_{|h|}} 
\frac{|f(x+h)
- f(x) |^{p}}{|h|^{\alpha p}} \ dx  \right)^{1/p}\;.$$
The local variant $N^{\alpha,p}_{\loc}(\Omega,\er^k )$ is defined by requiring that $f \in N^{\alpha,p}_{\loc}(\Omega,\er^k )$ iff $f \in N^{\alpha,p}(\tilde \Omega,\er^k)$ for every open subset $\tilde{\Omega} \Subset \Omega$.
\end{definition}
We have that $ W^{\alpha ,p}(\Omega,\er^k)\subsetneqq N^{\alpha,p}(\Omega,\er^k)\subsetneqq
W^{\beta,p}(\Omega,\er^k)$, for every $\beta <\alpha$, hold for sufficiently domains $\Omega$. A local, quantified version is in the next lemma (see for instance \cite{AKM}).
\begin{lemma}\label{l2}
Let $B_{r}\Subset \er^n$ be a ball with $r\leq 1$, $f\in L^{p}(B_{r},\mathbb{R}^{k})$, $p>1$ and assume that, for $\alpha \in (0,1]$, $S\ge 1$ and concentric balls $B_{\rr}\Subset B_{r}$, there holds
 \eqn{soddisfa}
 $$
\nr{\tau_{h}f}_{L^{p}(B_{\rr},\er^k)}\le S\snr{h}^{\alpha } \quad \mbox{
for every $h\in \mathbb{R}^{n}$ with $0<\snr{h}\le \frac{r-\rr}{K}$, where $K \geq 1$}\;.$$ 
Then $f\in W^{\beta,p}(B_{\rr},\mathbb{R}^{k})$ whenever $\beta\in (0,\alpha )$ and
\eqn{SSS}
$$
\nr{f}_{W^{\beta,p}(B_{\rr},\er^k)}\le\frac{c}{(\alpha -\beta)^{1/p}}
\left(\frac{r-\rr}{K}\right)^{\alpha -\beta}S+c\left(\frac{K}{r-\rr}\right)^{n/p+\beta} \nr{f}_{L^{p}(B_{r},\er^k)}\;,
$$
holds, where $c\equiv c(n,p)$. 
\end{lemma}
We finally report a well-known  iteration lemma whose proof can be found in \cite{GGActa}. 
\begin{lemma}\label{l0}
Let $\mathcal{Z}\colon [\rr,R)\to [0,\infty)$ be a function which is bounded on every interval $[\varrho, R_*]$ with $R_*<R$. Let $\varepsilon\in (0,1)$, $a_1,a_2,\gamma_{1},\gamma_{2}\ge 0$ be numbers. If
\begin{flalign*}
\mathcal{Z}(\tau_1)\le \varepsilon \mathcal{Z}(\tau_2)+ \frac{a_1}{(\tau_2-\tau_1)^{\gamma_{1}}}+\frac{a_2}{(\tau_2-\tau_1)^{\gamma_{2}}}\ \ \mbox{for all} \ \rr\le \tau_1<\tau_2< R\;,
\end{flalign*}
then
\begin{flalign*}
\mathcal{Z}(\rr)\le c\left[\frac{a_1}{(R-\rr)^{\gamma_{1}}}+\frac{a_2}{(R-\rr)^{\gamma_{2}}}\right]\;,
\end{flalign*}
holds with $c\equiv c(\varepsilon,\gamma_{1},\gamma_{2})$.
\end{lemma}

\section{Proof of Theorem \ref{main1}}

\subsection{A fractional Gagliardo-Nirenberg type inequality} In the proof of Theorem \ref{main1} we shall use a Gagliardo-Nirenberg type interpolation inequality, 
that we state here in a suitably localized form. In fact, the inequality we are going to use here requires the use of certain Gagliardo-Nirenberg inequalities in Triebel-Lizorkin spaces, as explained in \cite{brmi, brmi2}. 
\begin{lemma}\label{l3}
Let $B_{\rr}\Subset B_{r}\Subset \er^n$ be concentric balls with $r \leq 1$, $p,t\in (1,\infty)$, $s \in (1,2)$ and $f\in W^{s,p}(B_{r},\mathbb{R}^{N})\cap L^{2t}(B_{r},\mathbb{R}^{N})$ with $N\ge 1$. Then
\begin{flalign}\label{gn1}
\nr{f}_{W^{1,\tilde{p}}(B_{\rr})}\le \frac{c(n,p,s,t)}{(r-\rr)^{\kappa}}\nr{f}^{\frac{s-1}{s}}_{L^{2t}(B_{r})}\nr{f}^{\frac 1s}_{W^{s,p}(B_{r})}
\end{flalign}
holds with $\kappa \equiv \kappa (n,p,s,t)>0$, where
\begin{flalign}\label{51}
\tilde{p}:=\frac{2pst}{p(s-1)+2t}\;.
\end{flalign}
\end{lemma}
\begin{proof} We denote 
$s=1+\tau$, where $\tau\in (0,1)$; all the balls considered in the following will be concentric to $B_r$. Let $0<\rr<r\le 1$, $\eta\in C^{2}_{c}(B_{r})$ be a cut-off function such that
\begin{flalign}\label{lata1}
\mathds{1}_{B_{\rr}}\le \eta\le \mathds{1}_{B_{r_1}} \ \ \mbox{and} \ \ \snr{D\eta}^{2}+\snr{D^{2}\eta}\lesssim \frac{1}{(r-\rr)^{2}}\;,
\end{flalign}
where $r_1:=(r+\rr)/2$. From \cite[Lemma 3.1 and  Corollary 3.2, (a)]{brmi} (see also \cite{brmi2}) we know that, if $\tilde{f}\in W^{s,p}(\mathbb{R}^{n})\cap L^{2t}(\mathbb{R}^{n})$ with $p,s,t>1$, 
then there holds
\begin{flalign}\label{13}
\nr{\tilde{f}}_{W^{1,\tilde{p}}(\mathbb{R}^{n})}\le c(n,p,s,t)\nr{\tilde{f}}^{\frac{s-1}{s}}_{L^{2t}(\mathbb{R}^{n})}\nr{\tilde{f}}_{W^{s,p}(\mathbb{R}^{n})}^{\frac 1s}\;,
\end{flalign}
where $\tilde{p}$ is as in \eqref{51}. Let $\tilde{f}:=f\eta$, with $\eta$ being as in \eqref{lata1}. Let us check (recall \rif{derfrac}) that 
\begin{flalign}\label{15}
\tilde{f}\in W^{s,p}(\mathbb{R}^{n},\mathbb{R}^{N})\cap L^{2t}(\mathbb{R}^{n},\mathbb{R}^{N})\;.
\end{flalign}
We trivially have
\begin{flalign}\label{16}
\nr{\tilde{f}}_{L^{2t}(\mathbb{R}^{n})}\leq \nr{f}_{L^{2t}(B_{r})}\qquad \mbox{and}\qquad \nr{\tilde{f}}_{L^{p}(\mathbb{R}^{n})}\leq \nr{f}_{L^{p}(B_{r})}\;,
\end{flalign}
and (by $\eqref{lata1}$)
\eqn{17}
$$
\nr{D\tilde f}_{L^{p}(\er^n)}^{p}
\le \frac{c}{(r-\rr)^{p}}\nr{f}_{L^{p}(B_{r})}^{p}+c\nr{Df}^{p}_{L^{p}(B_{r})}\;.
$$
Next, set $r_2:=(\rr+3r)/4=(r_1+r)/2$, so that $\varrho < r_1 <r_2 <r$. Recalling that $\tilde f \equiv 0$ outside $B_{r_1}$, we have
\begin{flalign*}
[D\tilde{f}]_{\tau,r;\er^n}^p:=&\int_{\mathbb{R}^{n}}\int_{\mathbb{R}^{n}}\frac{\snr{D\tilde{f}(x)-D\tilde{f}(y)}^{p}}{\snr{x-y}^{n+\tau p}}  \dx \dy =\int_{B_{r_2}}\int_{B_{r_2}}\frac{\snr{D\tilde{f}(x)-D\tilde{f}(y)}^p}{\snr{x-y}^{n+\tau p}} \dx  \dy \nonumber \\
&+2\int_{\mathbb{R}^{n}\setminus B_{r_2}}\int_{B_{r_2}}\frac{\snr{D\tilde{f}(x)-D\tilde{f}(y)}^p}{\snr{x-y}^{n+\tau p}}  \dx \dy =:\mbox{(I)}+\mbox{(II)}\;.
\end{flalign*}
Expanding the expression of $\tilde{f}$, we have
\begin{eqnarray*}
\mbox{(I)}&\le &\, c\int_{B_{r_2}}\int_{B_{r_2}}\frac{\snr{\eta(x)Df(x)-\eta(y)Df(y)}^{p}}{\snr{x-y}^{n+\tau p}}\dx  \dy \nonumber \\
&&\, +c\int_{B_{r_2}}\int_{B_{r_2}}\frac{\snr{f(x)D\eta(x)-f(y)D\eta(y)}^{p}}{\snr{x-y}^{n+\tau p}}\dx  \dy =:c(p)\left[\mbox{(I)}_{1}+\mbox{(I)}_{2}\right]\;.
\end{eqnarray*}
Using also \rif{lata1}, we estimate
\begin{eqnarray*}
 \notag \mbox{(I)}_{1}& \leq & c\int_{B_{r_2}}\int_{B_{r_2}}\frac{[\eta(x)]^{p}\snr{Df(x)-Df(y)}^{p}}{\snr{x-y}^{n+\tau p}} \dx  \dy  
\\&& \quad + c \int_{B_{r_2}}\int_{B_{r_2}}\frac{\snr{Df(y)}^{p}\snr{\eta(x)-\eta(y)}^{p}}{\snr{x-y}^{n+\tau p}} \dx  \dy  \le c[Df]_{\tau, p;B_{r}}^p+\frac{cr^{p-p\tau}\nr{Df}^{p}_{L^{p}(B_{r})}}{(1-\tau)(r-\rr)^{p}}
\end{eqnarray*}
for $c\equiv c (n,p)$, and
\begin{eqnarray*}
\nonumber
\mbox{(I)}_{2}&\le &  c\int_{B_{r_2}}\int_{B_{r_2}}\frac{\snr{D\eta(x)}^{p}\snr{f(x)-f(y)}^{p}}{\snr{x-y}^{n+\tau p}} \dx  \dy \\ && \quad+c\int_{B_{r_2}}\int_{B_{r_2}}\frac{\snr{f(y)}^{p}\snr{D\eta(x)-D\eta(y)}^{p}}{\snr{x-y}^{n+\tau p}}  \dx \dy 
\le \frac{c[f]^{p}_{\tau,p;B_{r_2}}}{(r-\rr)^{p}}+\frac{cr^{p-p\tau}\nr{f}^{p}_{L^{p}(B_{r_2})}}{(1-\tau)(r-\rr)^{2p}}\;.
\end{eqnarray*}
Now notice that, if $h\in \mathbb{R}^{n}$ is any vector with $\snr{h}\le (r-r_2)/2= (r-\varrho)/8\leq 1$, since $f\in W^{1,p}(B_{r},\mathbb{R}^{N})$ there holds 
$$
\snr{h}^{-\frac{1+\tau}{2}}\nr{\tau_{h}f}_{L^{p}(B_{r_2})}\le c|h|^{\frac{1-\tau}{2}}\nr{Df}_{L^{p}(B_{r})}\le c(r-\varrho)^{\frac{1-\tau}{2}}\nr{Df}_{L^{p}(B_{r})}\;,
$$
so that Lemma \ref{l2} gives
$$
[f]_{\tau, p;B_{r_2}}\le \frac{c}{(1-\tau)^{\frac 1p}(r-\rr)^{\frac{n+p\tau}{p}}}\left(\nr{f}_{L^{p}(B_{r})}+\nr{Df}_{L^{p}(B_{r})}\right)\;,
$$
for $c\equiv c(n,p,\tau)$. Merging the content of the last three displays we obtain
$$
\mbox{(I)}_{2}\le \frac{c}{(1-\tau)(r-\rr)^{n+2p}}\left(\nr{f}_{L^{p}(B_{r})}^{p}+\nr{Df}^{p}_{L^{p}(B_{r})}\right)\;,
$$
with $c\equiv c(n,p,\tau)$ and we have used that $r \leq 1$. As for (II), we have
\begin{eqnarray*}
\mbox{(II)}&=&2\int_{\mathbb{R}^{n}\setminus B_{r_2}}\int_{B_{r_2}}\frac{\snr{D\tilde{f}(x)}^p}{\snr{x-y}^{n+\tau p}}  \dx \dy
\nonumber \\
&\leq &2\int_{\mathbb{R}^{n}\setminus B_{r_2}}\int_{B_{r_1}}\frac{|\eta(x)Df(x)|^{p}}{\snr{x-y}^{n+\tau p}} \dx  \dy 
+2\int_{\mathbb{R}^{n}\setminus B_{r_2}}\int_{B_{r_1}}\frac{|f(x)D\eta(x)|^{p}}{\snr{x-y}^{n+\tau p}} \dx  \dy \\
&=: &   \mbox{(II)}_{1}+\mbox{(II)}_{2}\;,
\end{eqnarray*}
and we have used that $\eta$ vanishes outside $B_{r_1}$. As $r_2-r_1=(r-\rr)/4$, note that if $x\in B_{r_1}$ and $y\in \mathbb{R}^{n}\setminus B_{r_2}$, then
\begin{flalign*}
\snr{y-x}\ge \snr{y}\left[\frac{\snr{y}-\snr{x}}{\snr{y}}\right]=\snr{y}\left[1-\frac{\snr{x}}{\snr{y}}\right]\ge \frac{\snr{y}}{4r_2}(r-\rr)\;.
\end{flalign*}
Using this fact we estimate as follows:
\begin{eqnarray*}
\mbox{(II)}_{1}&\le & 2\int_{\mathbb{R}^{n}\setminus B_{r_2}}\int_{B_{r_1}}\frac{\snr{Df(x)}^{p}}{\snr{x-y}^{n+\tau p}} \dx  \dy \le \frac{cr_2^{n+\tau p}}{(r-\rr)^{n+\tau p}}\int_{\mathbb{R}^{n}\setminus B_{r_2}}\int_{B_{r_1}}\frac{\snr{Df(x)}^{p}}{\snr{y}^{n+\tau p}}  \dx \dy \nonumber \\
&\le&\frac{cr_2^{n}}{\tau(r-\rr)^{n+\tau p}}\nr{Df}_{L^{p}(B_{r})}^{p}
\leq \frac{c(n,p)}{\tau(r-\rr)^{n+\tau p}}\nr{Df}_{L^{p}(B_{r})}^{p}\;,
\end{eqnarray*}
and, as before, but also using \rif{lata1}, we get
\begin{eqnarray*}
\nonumber \mbox{(II)}_{2}&\le &\frac{c}{(r-\rr)^{p}}\int_{\mathbb{R}^{n}\setminus B_{r_2}}\int_{B_{  r_1}}\frac{\snr{f(x)}^{p}}{\snr{x-y}^{n+\tau p}} \dx  \dy \\
&\le & \frac{cr_2^{n+\tau p}}{(r-\rr)^{n+p(1+\tau)}}\int_{\mathbb{R}^{n}\setminus B_{r_2}}\int_{B_{r_1}}\frac{\snr{f(x)}^{p}}{\snr{y}^{n+\tau p}} \dx  \dy 
\le\frac{c}{\tau (r-\rr)^{n+2 p}}\nr{f}^{p}_{L^{p}(B_{r})}\;,
\end{eqnarray*}
where, in  both inequalities, it is $c\equiv c (n,p)$.
Collecting the estimates found for the terms $(I)_1, (I)_2,(II)_1,(II)_2$, and recalling \rif{16}-\rif{17}, we conclude with
\eqn{viavia}
$$
\nr{\tilde f}_{W^{s,p}(\er^n)} \le \frac{c(n,p,s,r)}{(r-\rr)^{\kappa}}\nr{f}_{W^{s,p}(B_{r})}\;,
$$
for a constant $c,\kappa$ depending as in \rif{gn1}. This proves \rif{15}. Finally, we have
$$
\nr{f}_{W^{1,\tilde{p}}(B_{\rr})}\stackrel{\eqref{lata1}}{\le}\nr{\tilde{f}}_{W^{1,\tilde{p}}(\mathbb{R}^{n})}\stackrel{\eqref{13}}{\le} c\nr{\tilde{f}}^{\frac{s-1}{s}}_{L^{2t}(\mathbb{R}^{n})}\nr{\tilde{f}}_{W^{s,p}(\mathbb{R}^{n})}^{\frac 1s}
$$
from which \rif{gn1} follows using \rif{16}$_1$ and \rif{viavia}. 
\end{proof}

\subsection{Theorem \ref{main1}, case $p\geq 2$} 
{\em Step 1: Convergence}. We take a ball $B_R \Subset \Omega$ with $R\leq 1$, as in the statement of Theorem \ref{main1}, i.e., such that \rif{zerogap} holds. Fix $\beta \in (0, \alpha)$ arbitrarily. If we prove \rif{main-ine} whenever $\tilde p$ is such that $q < \tilde p < p+\beta$ we have finished. Moreover, it is sufficient to prove \rif{main-ine} for numbers of the form 
\eqn{23bis}
$$
q< \tilde{p}\equiv \tilde p (t):=\frac{2t(p+\beta)}{2t+\beta} \qquad \mbox{with $t \geq 2p$}
$$
since $\tilde p (t)\to p+\beta$ as $t\to \infty$. Therefore from now on we fix an arbitrary number $t$ satisfying \rif{23bis}. 
Notice that the first condition in \rif{23bis} implies
\eqn{23bis22} 
$$
q< p + \frac{\beta(2t-p)}{2t}\;.
$$
Now we combine and modify the approximation arguments considered in \cite{CKP} and \cite{sharp}. 
Proposition \ref{approssi2} and the lower bound in \rif{assF}$_1$ yield the existence of a sequence $\{\tilde{u}_{j}\} \subset W^{1,q}(B_{R},\RN)$ such that
\begin{flalign}\label{60}
\mathcal{F}(\tilde{u}_{j},B_{R})\to\mathcal{F}(u,B_{R}), \ \ \tilde{u}_{j}\rightharpoonup u \ \mbox{in} \ W^{1,p}(B_{R},\RN), \ \ \nr{\tilde{u}_{j}-u}_{L^{\infty}(B_{R})}\to0\;.
\end{flalign}
Let us define $u_{j}\in \tilde{u}_{j}+W^{1,q}_{0}(B_{R},\RN)\cap L^{2t}(B_{R},\RN)$ as the solution of the Dirichlet problem
\eqn{solva}
$$
 u_j\mapsto \min_{w \in \tilde{u}_{j}+W^{1,q}_{0}(B_{R},\RN)} \mathcal{F}_{j}(w,B_R)\;,
$$
where, denoting as usual
$
(\snr{w}^{2}-M^{2})_{+} = \max\{\snr{w}^{2}-M^{2}, 0\},
$
it is
\eqn{FJ}
$$
\mathcal{F}_{j}(w, B_R):=\int_{B_{R}}\left[F(x,Dw)+(\snr{w}^{2}-M^{2})_{+}^{t}\right]  \dx+\frac{\varepsilon_{j}}{q}\int_{B_{R}}(1+\snr{Dw}^{2})^{\frac{q}{2}}  \dx\;,
$$
with 
\begin{flalign}\label{24}
\varepsilon_{j}:=\left(1+j+\nr{D\tilde{u}_{j}}_{L^{q}(B_{R})}^{2q}\right)^{-1}
\end{flalign}
and
\begin{flalign}\label{23}
M:=2\nr{u}_{L^{\infty}(B_{R})}+1\;.
\end{flalign}
Notice that \eqref{24} guarantees 
\begin{flalign}\label{25}
\frac{\varepsilon_{j}}{q}\int_{B_{R}}(1+\snr{D\tilde{u}_{j}}^{2})^{\frac{q}{2}}  \dx\to0\;.
\end{flalign}
Notice also that the (unique) solvability of \rif{solva} follows by Direct Methods of the Calculus of Variations and convexity. 
As a consequence of \eqref{60}, there exists $\tilde{j} \in \en$ such that
\begin{flalign}\label{27}
\nr{\tilde{u}_{j}}_{L^{\infty}(B_{R})}\le 2\nr{u}_{L^{\infty}(B_{R})} \ \ \mbox{for} \ j\ge \tilde{j}\ge 1\;.
\end{flalign}
Up to relabelling the sequence $\{\tilde{u}_{j}\}$, we can take $\tilde{j}=1$. By minimality, \eqref{60}, $\eqref{23}$, \eqref{25} and \eqref{27} it is easy to see that
\begin{flalign}\label{26}
\limsup_{j\to \infty}\mathcal{F}_{j}(u_{j}, B_R)\le \limsup_{j\to \infty}\mathcal{F}_{j}(\tilde{u}_{j}, B_R)=\mathcal{F}(u,B_{R})\;.
\end{flalign}
Moreover, \eqref{26} and $\eqref{assF}_{1}$ yield that the sequence $\{Du_j\}$ is bounded in $L^p(B_R)$. 
Up to not relabelled subsequences, we then get that
\begin{flalign}\label{28}
u_{j}\rightharpoonup v \  \mbox{in}  \ W^{1,p}(B_{R},\RN) \ \ \mbox{for some} \ \ v\in u+W^{1,p}_{0}(B_{R},\RN)\;.
\end{flalign}
By weak lower semicontinuity we have that
\begin{eqnarray*}
\nonumber \mathcal{F}(u,B_{R})&\stackrel{\eqref{26}}{\ge} &\liminf_{j\to \infty}\mathcal{F}_{j}(u_{j}, B_R)\\ 
&\geq &\liminf_{j\to \infty} \int_{B_{R}}\left[F(x,Du_j)+(\snr{u_j}^{2}-M^{2})_{+}^{t}\right]  \dx\\
&\stackrel{\eqref{28}} {\geq}&\int_{B_{R}}\left[F(x,Dv)+(\snr{v}^{2}-M^{2})_{+}^{t}\right]  \dx\ge \mathcal{F}(v,B_{R})\;.
\end{eqnarray*}
As $u-v \in W^{1,1}_0(B_R)$, minimality of $u$ yields 
$ \mathcal{F}(u,B_{R})\leq  \mathcal{F}(v,B_{R})$ and therefore $ \mathcal{F}(v,B_{R})=\mathcal{F}(u,B_{R})$. The strict convexity of $z\mapsto F(\cdot,z)$ (implied by \rif{assF}$_2$), then leads to 
$u=v$, so that again lower semicontinuity yields
$$
\mathcal{F}(u,B_{R})\le \limsup_{j\to \infty}\mathcal{F}(u_{j}, B_R)\leq\limsup_{j\to \infty}\mathcal{F}_{j}(u_{j}, B_R)\stackleq{26}  \mathcal{F}(u,B_{R})\;,
$$
and therefore we conclude with
$$
\lim_{j\to \infty}\int_{B_{R}}(\snr{u_{j}}^{2}-M^{2})_{+}^{t}  \dx=0 
$$
that, in turn, implies
\eqn{32}
$$\sup_{j\in \mathbb{N}}\, \nr{u_{j}}_{L^{p}(B_{R})}+  \sup_{j\in \mathbb{N}}\, \nr{u_{j}}_{L^{2t}(B_{R})}<c\left(n,t,\nr{u}_{L^{\infty}(B_{R})}\right)\;,$$
where we also used the explicit expression of $M$ reported in \eqref{23} and that $p\leq 2t$, $R\leq 1$.

{\em Step 2: A priori estimates.} We use the short notation 
$$
F_j(x,z) := F(x,z) + \frac{\varepsilon_{j}}{q}(1+\snr{z}^{2})^{\frac{q}{2}} \;.
$$
The Euler-Lagrange equation of the functional $\mathcal F_j$ in \rif{FJ} reads
\begin{flalign}\label{EL}
\int_{B_{R}}\left[\partial_z F_{j}(x,Du_{j})\cdot D\varphi+2t(\snr{u_{j}}^{2}-M^{2})_{+}^{t-1}u_{j}\cdot \varphi \right] \dx=0
\end{flalign}
and holds whenever $\varphi\in W^{1,q}_{0}(B_{R},\RN)\cap L^{2t}(B_{R},\RN)$ as $u_j \in W^{1,q}(B_R,\er^N)\cap L^{2t}(B_{R},\RN)$ and $F_j(\cdot)$ has $q$-growth conditions with respect to the gradient variable.  Notice that the integrands still $F_j(\cdot)$ satisfy  the following monotonicity inequality:
\eqn{usa1}
$$
 \left(\partial_z F_{j}(x,z_2)-\partial_z F_{j}(x,z_1)\right) \cdot (z_2-z_2) \geq \nu
(\lambda^{2}+\snr{z_2}^{2}+\snr{z_1}^{2})^{\frac{p-2}{2}}\snr{z_2-z_1}^{2}\;,
$$
for every $z_1, z_2 \in \er^{N\times n}$, where $c\equiv c (n,p,q)$. 
This is a straightforward consequence of the ellipticity assumption \rif{assF}$_2$ (see for instance \cite{KM3}). Now, fix $0<\rr\le \tau_{1}<\tau_{2} < R$ and set $\varphi:=\tau_{-h}(\eta^{2}\tau_{h}u_{j})$, which is admissible in \eqref{EL}, as we take
\eqn{iltaglio}
$$
\eta\in C^1_{\rm{c}}(B_{\frac{3\tau_{2}+\tau_{1}}{4}}), \ \ \mathds{1}_{B_{\left(\tau_{1}+\tau_{2}\right)/2}}\le \eta \le \mathds{1}_{B_{\frac{3\tau_{2}+\tau_{1}}{4}}}, \ \ \snr{D\eta}\lesssim \frac{1}{(\tau_{2}-\tau_{1})}
$$
and $h\in \mathbb{R}^{n}\setminus \{0\}$ is any fixed vector with $\snr{h}<\frac{\tau_{2}-\tau_{1}}{1024}\leq 1$. Testing \eqref{EL} with $\varphi$ and using the integration by parts formula for finite difference operators, we obtain
\begin{flalign*}
0=&\int_{B_{R}}\tau_{h}\left(\partial_z F_{j}(x,Du_{j})\right)\cdot (2\eta D\eta\otimes \tau_{h}u_{j}+\eta^{2}\tau_{h}Du_{j})  \dx\\
&\quad +2t\int_{B_{R}}\eta^{2}\tau_{h}((\snr{u_{j}}^{2}-M^{2})_{+}^{t-1}u_{j})\cdot\tau_{h}u_{j}  \dx=:\mbox{(I)}_{j}+\mbox{(II)}_{j}\;.
\end{flalign*}
For $(I)_{j}$, we decompose
\begin{eqnarray}
\notag &&\mbox{(I)}_{j}=
\\ 
\notag &&\int_{B_{R}}\left(\partial_z F_{j}(x+h,Du_{j}(x+h))-\partial_z F_{j}(x+h,Du_{j}(x))\right)\cdot (\eta^{2}\tau_{h}Du_{j}+2\eta D\eta\otimes \tau_{h}u_{j}) \dx\\
\notag 
&&\ \ + \int_{B_{R}}\left(\partial_z F_{j}(x+h,Du_{j}(x))-\partial_z F_{j}(x,Du_{j}(x))\right)\cdot(2\eta D\eta\otimes \tau_{h}u_{j}+\eta^{2}\tau_{h}Du_{j})  \dx\\
&& \quad \ \  =:\mbox{(I)}_{j}^{1}+\mbox{(I)}_{j}^{2}+\mbox{(I)}_{j}^{3}+\mbox{(I)}_{j}^{4}\;,\label{alfa1}
\end{eqnarray}
with obvious meaning of the notation. We have
\eqn{alfa2}
$$
\mbox{(I)}_{j}^{1} \stackrel{\rif{usa1}}{\geq}  \nu \int_{B_{R}}\eta^{2}(\lambda^{2}+\snr{Du_{j}(x+h)}^{2}+\snr{Du_{j}(x)}^{2})^{\frac{p-2}{2}}\snr{\tau_{h}Du_{j}}^{2}  \dx$$
while, using also H\"older  inequality, we find
\begin{eqnarray}
&& 
\nonumber \snr{\mbox{(I)}_{j}^{2}} + \snr{\mbox{(I)}_{j}^{3}}
 \stackrel{\rif{crescitader}}{\leq}   c\int_{B_{R}}\eta(1+\snr{Du_{j}(x+h)}^{2}+\snr{Du_{j}(x)}^{2})^{\frac{q-1}{2}}\snr{D\eta}\snr{\tau_{h}u_{j}}  \dx\\
\nonumber
&& \qquad \leq   c\|D\eta\|_{L^\infty}\left(\int_{B_{R}}\eta(1+\snr{Du_{j}(x+h)}^{2}+\snr{Du_{j}(x)}^{2})^{\frac{q}{2}} \dx\right)^{\frac{q-1}{q}}
\left(\int_{B_{R}}\eta\snr{\tau_{h}u_{j}}^q  \dx\right)^{\frac{1}{q}}\\
&&\qquad \leq  c  \|D\eta\|_{L^\infty}|h|\int_{B_{R}}(1+\snr{Du_{j}}^{2})^{\frac{q}{2}} \dx \;, \label{alfa3}
\end{eqnarray}
where $c \equiv c (n,N,p,q)$. 
Here we have used a standard property of finite difference operators, i.e., by \rif{iltaglio} and as $\snr{h}<\frac{\tau_{2}-\tau_{1}}{1024}$, it holds that
$$
\int_{B_{R}}\eta\snr{\tau_{h}u_{j}}^q  \dx \leq \int_{B_{(3\tau_{2}+\tau_{1})/4}}\snr{\tau_{h}u_{j}}^q  \dx \leq c|h|^q \int_{B_{R}}\snr{Du_{j}}^q \dx\;.
$$
Finally, we have
$$
\snr{\mbox{(I)}_{j}^{4}}
 \stackrel{\eqref{crescitader}_3}{\leq}   c|h|^\alpha\int_{B_{R}}(1+|Du|^2)^{\frac q2}  \dx\;,
$$
and in  both the last inequalities it is $c\equiv c (n,N,L, p,q)$. 
As for the term $\mbox{(II)}_{j}$, we have
\begin{eqnarray}
\mbox{(II)}_{j}&=&2t\int_{B_{R}}\eta^{2}\int_{0}^{1}\frac{d}{d\theta}\left((\snr{u_{j}+\theta\tau_{h}u_{j}}^{2}-M^{2})_{+}^{t-1}(u_{j}+\theta\tau_{h}u_{j})\right) \, \d \theta \cdot  \tau_{h}u_{j}  \dx\nonumber \\
&=&2t\int_{B_{R}}\eta^{2}\int_{0}^{1}\left(2(t-1)(\snr{u_{j}+\theta \tau_{h}u_{j}}^{2}-M^{2})^{t-2}_{+}(u_{j}+\theta \tau_{h}u_{j})\otimes(u_{j}+\theta \tau_{h}u_{j})\right.\nonumber \\
&&\qquad \qquad \qquad +\left.(\snr{u_{j}+\theta \tau_{h}u_{j}}^{2}-M^{2})_{+}^{t-1}\right) \, \d\theta \  \tau_{h}u_{j}\cdot\tau_{h}u_{j}  \dx\nonumber \\
&=&2t\int_{B_{R}}\eta^{2}\int_{0}^{1}2(t-1)(\snr{u_{j}+\theta \tau_{h}u_{j}}^{2}-M^{2})^{t-2}_{+}((u_{j}+\theta \tau_{h}u_{j})\cdot \tau_{h}u_{j})^{2} \, \d \theta  \dx\nonumber\\
&&+2t\int_{B_{R}}\eta^{2}\int_{0}^{1}(\snr{u_{j}+\theta \tau_{h}u_{j}}^{2}-M^{2})^{t-1}_{+} \, \d\theta \ \snr{\tau_{h}u_{j}}^{2}  \dx \ge 0\;.\label{alfa4}
\end{eqnarray}
Connecting the estimates in \rif{alfa1}-\rif{alfa4}, and again using \rif{iltaglio}, yields
\begin{align}
& \int_{B_{R}}\eta^{2}(\lambda^{2}+\snr{Du_{j}(x+h)}^{2}+\snr{Du_{j}(x)}^{2})^{\frac{p-2}{2}}\snr{\tau_{h}Du_{j}}^{2}  \dx \nonumber \\
 & \qquad \le \frac{c\snr{h}^{\alpha}}{\tau_{2}-\tau_{1}}\int_{B_{\tau_{2}}}(1+\snr{Du_{j}}^{2})^{\frac{q}{2}}   \dx\;,\label{33}
\end{align}
for $c\equiv c(n,N,\nu,L,p,q)$; we have used that $|h|\leq |h|^\alpha$ as it is $|h|\leq 1$. The last estimate is valid whenever $p>1$. As we are considering the case $p\geq 2$, \rif{33} implies
\eqn{formally}
$$
\int_{B_{\left(\tau_{1}+\tau_{2}\right)/2}}\snr{\tau_{h}Du_{j}}^{p}  \dx 
\le \frac{c\snr{h}^{\alpha}}{\tau_{2}-\tau_{1}}\int_{B_{\tau_{2}}}(1+\snr{Du_{j}}^{2})^{\frac{q}{2}}   \dx\;,
$$
for $c\equiv c(n,N,\nu,L,p,q)$ and this holds whenever $h\in \mathbb{R}^{n}\setminus \{0\}$ is such that $\snr{h}<\frac{\tau_{2}-\tau_{1}}{1024}$. The content of \rif{formally} allows to satisfy \rif{soddisfa}, and then \eqref{32} and Lemma \ref{l2} give that
$$
u_{j}\in W^{1+\beta/p,p}(B_{\left(\tau_{1}+\tau_{2}\right)/2},\RN)\cap L^{2t}(B_{R},\RN) \ \ \mbox{holds for all} \ \beta\in (0,\alpha)
$$
with
\begin{eqnarray}\label{36}
\nonumber \nr{u_{j}}_{W^{1+\beta/p,p}(B_{\left(\tau_{1}+\tau_{2}\right)/2},\er^N)}&\stackleq{SSS}& \nr{u_{j}}_{L^{p}(B_{\tau_{2}})}+\frac{c}{(\tau_{2}-\tau_{1})^{(n+\beta)/p}}\left(1+\nr{Du_{j}}_{L^{q}(B_{\tau_{2}})}^{q/p}
\right)\\
&\stackleq{32} &\frac{c}{(\tau_{2}-\tau_{1})^{(n+\beta)/p}}\left(1+\nr{Du_{j}}_{L^{q}(B_{\tau_{2}})}^{q/p}\right)
\;,
\end{eqnarray}
where $c\equiv c(n,N,\nu,L,p,q,\alpha,\beta,\|u\|_{L^{\infty}(B_{R})})$; we are again using that $R\leq 1$. We now use Lemma \ref{l3} with 
\eqn{pipipi}
$$s = 1+\frac{\beta}{p}\qquad \mbox{and}\qquad \tilde{p}=\frac{2t(p+\beta)}{2t+\beta} \ \  \mbox{(as taken in \rif{23bis})}\;,$$ thereby getting
\begin{eqnarray}
\notag \nr{Du_{j}}_{L^{\tilde{p}}(B_{\tau_{1}})}&\stackleq{gn1}&\frac{c}{(\tau_{2}-\tau_{1})^{\kappa}}\nr{u_{j}}_{L^{2t}(B_{\left(\tau_{1}+\tau_{2}\right)/2})}^{\frac{\beta}{p+\beta}}\nr{u_{j}}^{\frac{p}{p+\beta}}_{W^{1+\beta/p,p}(B_{(\tau_{1}+\tau_{2})/2})}\\
&\stackrel{\rif{32},\rif{36}}{\leq}&\frac{c}{(\tau_{2}-\tau_{1})^{\kappa+\frac{n+\beta}{p+\beta}}}\left(1+\nr{Du_{j}}_{L^{q}(B_{\tau_{2}})}^{\frac{q}{p+\beta}}
\right)
\label{dipet}
\end{eqnarray}
for $c\equiv c(n,N,\nu,L,p,q,\alpha,\beta,\|u\|_{L^{\infty}(B_{R})})$ and $ \kappa \equiv   \kappa(n,p,\beta,t)$. Notice that as $\tau_1$ and $\tau_2$ have been chosen arbitrarily, we have proved that
\eqn{finita}
$$
Du_j \in L^{\tilde p}(B_{\tau_1}) \qquad \mbox{for every $\tau_1 < R$}\;.
$$ 
Now we commute this into a uniform a priori estimate with respect to $j$. 
Thanks to the first inequality in $\eqref{23bis}$ we can interpolate with the inequality
\begin{flalign}\label{interp}
\nr{Du_{j}}_{L^{q}(B_{\tau_{2}})}\le \nr{Du_{j}}^{\tilde{\theta}}_{L^{\tilde{p}}(B_{\tau_{2}})}\nr{Du_{j}}^{1-\tilde{\theta}}_{L^{p}(B_{\tau_{2}})}\;, \quad  \qquad \frac{1}{q}=\frac{\tilde{\theta}}{\tilde{p}}+\frac{1-\tilde{\theta}}{p}\;, 
\end{flalign}
that is
\eqn{ttheta}
$$
\tilde \theta = \frac{(q-p)\tilde p}{(\tilde p-p)q}\in (0,1)\;.
$$
Plugging \rif{interp} in \rif{dipet} gives
\eqn{37}
$$
\nr{Du_{j}}_{L^{\tilde{p}}(B_{\tau_{1}})}\le \,\frac{c}{(\tau_{2}-\tau_{1})^{\kappa+\frac{n+\beta}{p+\beta}}} \left(1+\nr{Du_{j}}^{\frac{\tilde{\theta}q}{p+\beta}}_{L^{\tilde{p}}(B_{\tau_{2}})}\nr{Du_{j}}^{\frac{(1-\tilde{\theta})q}{p+\beta}}_{L^{p}(B_{\tau_{2}})}\right)\;,
$$
for $c\equiv c(n,N,\nu,L,p,q,\alpha,\beta,t, \|u\|_{L^{\infty}(B_{R})})$. 
Notice that \rif{ttheta} implies
$$
\frac{\tilde{\theta}q}{p+\beta}< 1 \Longleftrightarrow  q-p < \frac{\beta(2t-p)}{2t}\;,$$
and the last inequality is satisfied by \rif{23bis22}. 
Therefore applying Young inequality yields
\eqn{377}
$$
\nr{Du_{j}}_{L^{\tilde{p}}(B_{\tau_{1}})}\le  \frac{1}{2}\nr{Du_{j}}_{L^{\tilde{p}}(B_{\tau_{2}})}+\frac{c}{(\tau_{2}-\tau_{1})^{\kappa_1}}\left(1+\nr{Du_{j}}_{L^{p}(B_{R})}\right)^{p\kappa_2}
$$
for a constant $c$ depending as in \rif{37} and exponents $\kappa_1,\kappa_2\equiv \kk_1, \kk_2 (n,p,q,\alpha,\beta,t)>1$. 
This holds whenever $\varrho \leq \tau_1 < \tau_2 < R$. 
Inequality \rif{377} allows to apply Lemma \ref{l0} with the choice 
$\mathcal Z(\mathcal{l})\equiv \nr{Du_{j}}_{L^{\tilde{p}}(B_{\mathcal{l}})}$. This is by \rif{finita} a bounded function on every interval $[\varrho, R_*]$ whenever $\varrho < R_* <R$. We obtain
$$
\nr{Du_{j}}_{L^{\tilde{p}}(B_{\varrho})}\le \frac{c}{(R-\varrho)^{\kappa_1}}\left(1+\nr{Du_{j}}_{L^{p}(B_{R})}\right)^{p\kappa_2}
$$
where $c\equiv c(n,N,\nu,L,p,q,\alpha,\beta,t,\nr{u}_{L^{\infty}(B_{R})})$, and using \rif{assF}$_1$ we find
$$
\nr{Du_{j}}_{L^{\tilde{p}}(B_{\rr})}\le \frac{c}{(R-\rr)^{\kappa_{1}}}\left[1+\mathcal{F}_j(u_j, B_R)\right]^{\kappa_{2}}\;.
$$
Recalling \eqref{28} (and that $u=v$) and \rif{26}, letting $j \to \infty$ in the above display 
yields \rif{main-ine} via lower semicontinuity.

\subsection{Theorem \ref{main1}, case $1<p<2$} The proof largely proceeds as in the case $p\geq 2$ and we confine ourselves to describe the relevant modifications. We fix $\beta\in (0,\alpha)$ and prove \rif{main-ine} for $q < \tilde p < p+p\beta /2$. We this time take $\tilde p$ of the form  
\eqn{pipipi2}
$$
q<\tilde{p}:=\frac{2tp(2+\beta)}{4t+p\beta}
$$
for $t \geq 2$ and observe that this implies 
\eqn{condi}
$$
q<p+ \frac{p\beta(2t-p)}{4t}\;.
$$
With this new choice of the number $t$ the proof proceeds as for the case $p\geq 2$, up to \rif{33}. 
As now it is $p<2$, H\"older inequality gives
\begin{align*}
\int_{B_{\left(\tau_{1}+\tau_{2}\right)/2}}\snr{\tau_{h}Du_{j}}^{p} \dx  \leq & \left(\int_{B_{\left(\tau_{1}+\tau_{2}\right)/2}}(\lambda^{2}+\snr{Du_{j}(x+h)}^{2}+\snr{Du_{j}(x)}^{2})^{\frac{p-2}{2}}\snr{\tau_{h}Du_{j}}^{2} \dx\right)^{\frac{p}{2}}
\\
&\ \  \cdot \left(\int_{B_{\left(\tau_{1}+\tau_{2}\right)/2}}(\lambda^{2}+\snr{Du_{j}(x+h)}^{2}+\snr{Du_{j}(x)}^{2})^{\frac{p}{2}} \dx\right)^{\frac{2-p}{2}}\;.
\end{align*}
Notice that here, as well as in \rif{33}, we are using the standard and obvious convention to interpret all the quantities involving $Du_j$ as zero at those points where $\snr{Du_{j}(x+h)}=\snr{Du_{j}(x)}=0$; this remark is necessary only when $\lambda=0$. 
Using \rif{33} in the above inequality easily leads to 
$$
\int_{B_{\left(\tau_{1}+\tau_{2}\right)/2}}\snr{\tau_{h}Du_{j}}^{p}  \dx 
\le \frac{c\snr{h}^{p\alpha/2}}{(\tau_{2}-\tau_{1})^{p/2}}\int_{B_{\tau_{2}}}(1+\snr{Du_{j}}^{2})^{\frac{q}{2}}   \dx\;,
$$
which is formally analogous to \rif{formally}. Therefore, proceeding as for the case $p\geq 2$ and applying Lemma \ref{l2}, we get
\eqn{weget}
$$
\nr{u_{j}}_{W^{1+\beta/2,p}(B_{\left(\tau_{1}+\tau_{2}\right)/2})}\le \frac{c}{(\tau_{2}-\tau_{1})^{n/p+\beta/2}}\left(1+\nr{Du_{j}}_{L^{q}(B_{\tau_{2}})}^{q/p}
\right)\;.
$$
Again we apply Lemma \ref{13} with the new parameters $s = 1+\beta/2$ and $\tilde p$ in \rif{pipipi2} obtaining
\begin{eqnarray}
\notag \nr{Du_{j}}_{L^{\tilde{p}}(B_{\tau_{1}})}&\stackleq{gn1}&\frac{c}{(\tau_{2}-\tau_{1})^{\kappa}}\nr{u_{j}}_{L^{2t}(B_{\left(\tau_{1}+\tau_{2}\right)/2})}^{\frac{\beta}{2+\beta}}\nr{u_{j}}^{\frac{2}{2+\beta}}_{W^{1+\beta/2,p}(B_{(\tau_{1}+\tau_{2})/2})}\\
&\stackrel{\rif{32},\rif{weget}}{\leq}&\frac{c}{(\tau_{2}-\tau_{1})^{\kappa+ \frac{2n+p\beta}{p(2+\beta)}}}\left(1+\nr{Du_{j}}_{L^{q}(B_{\tau_{2}})}^{\frac{2q}{p(2+\beta)}}
\right)\;,
\label{dipet2}
\end{eqnarray}
for $c\equiv c(n,N,\nu,L,p,q,\alpha,\beta,\|u\|_{L^{\infty}(B_{R})})$ and $ \kappa \equiv   \kappa(n,p,\beta,t)$. 
We then interpolate exactly as in \rif{interp}-\rif{ttheta}; plugging \rif{interp} in the above inequality yields
\eqn{37bis}
$$
\nr{Du_{j}}_{L^{\tilde{p}}(B_{\tau_{1}})}\le \frac{c}{(\tau_{2}-\tau_{1})^{\kappa+ \frac{2n+p\beta}{p(2+\beta)}}}\left(1+\nr{Du_{j}}^{\frac{2\tilde{\theta}q}{p(2+\beta)}}_{L^{\tilde{p}}(B_{\tau_{2}})}\nr{Du_{j}}^{\frac{2(1-\tilde{\theta})q}{p(2+\beta)}}_{L^{p}(B_{\tau_{2}})}\right)\;,
$$
where $\tilde{\theta}$ is as in \rif{ttheta}, but with the new expression of $\tilde p$ defined in \rif{pipipi2}. We then observe that this time it is
$$
\frac{2\tilde{\theta}q}{p(2+\beta)}< 1 \Longleftrightarrow  q-p <\frac{p\beta(2t-p)}{4t}\;.
$$
The last inequality is the one in \rif{condi} and therefore we can proceed as after \rif{37} in the case $p\geq 2$. The proof of Theorem \ref{main1} is complete. 
\section{Proof of Theorem \ref{main3}}
\subsection{Step 1: Initial approximation}\label{apri} We immediately observe that, up to passing to the new integrand $F(\cdot)/\nu$, we can assume it is $\nu=1$ in \rif{assF2}. Indeed this new integrand satisfies assumptions \rif{assF2} with $\nu=1$ and $L$ replaced by $L/\nu$. Let $B_R\Subset \Omega $ be a ball as in the statement of Theorem \ref{main3}, i.e., $R\leq 1$ and \rif{zerogap2} holds; as in the statement, we also fix a concentric ball $B_{\varrho}\Subset B_R$. We consider a standard family of symmetric mollifiers $\{\phi_{\delta}\}_{\delta}$ for $\delta >0$ such that $\delta < \min\{\dist(B_R,\partial \Omega),1\}/8$, that is
\eqn{mollificatori}
$$\phi \in C^{\infty}_{\rm{c}}(B_1(0))\,,\quad \|\phi\|_{L^1} =1\,,\quad \phi_{\delta}(x)\equiv  \delta^{-n}\phi\left(x/\delta\right)\,,\quad B_{3/4}\subset \textnormal{supp}\, \phi \,.$$ 
Notice that $B_{R+\delta}\Subset \Omega$. We then define
\eqn{molli1}
$$
F_{\delta}(x,z):=(F*\phi_{\delta})(x,z)=\mint_{B_{1}}\mint_{B_{1}}F(x+\delta \tilde y,z+\delta y) \phi(\tilde{y})  \phi(y)\d\tilde{y} \dy 
\;,
$$
for all $(x,z)\in \overline{B_R}\times \er^n$. 
By the very definition in \rif{molli1} and \rif{holderF}, we have
\eqn{81} 
$$
F_{\delta}(x,z) \to F(x,z) \quad \mbox{uniformly on compact subsets of $\overline{B_R}\times \er^n$ as $\delta \to 0$}\;.
$$
We further define 
\eqn{molli2}
$$
h_{\delta}(x):=(h*\phi_{\delta})(x)=\mint_{B_{1}} h(x+\delta \tilde y)\phi(\tilde{y}) \d\tilde{y} \,, \quad \lambda_{\delta}:=\lambda+\delta\,, \quad H_{\delta}(z):=\lambda^{2}_{\delta}+\snr{z}^{2}\;,
$$
for $x \in \overline{B_R}$ and $z\in \er^n$. 
Next, we use assumption \rif{zerogap2}, that is
$\mathcal{L}_{b}^{q}(u,B_R)=0$ (see the definition in \rif{lav2}), and Proposition \ref{approssi1}, to get the existence of a sequence $\{\tilde{u}_{j}\} \subset W^{1,q}(B_{R})$ such that 
\begin{flalign}\label{73}
\tilde{u}_{j}\rightharpoonup u \ \mbox{in} \ W^{1,1}(\Omega,\er^N) \ \ \mbox{and} \ \ \mathcal{F}(\tilde{u}_{j},B_{R})\to \mathcal{F}(u,B_{R})\;.
\end{flalign}
We then set, for $(x, z) \in \overline{B_R}\times \er^n$,
\begin{flalign*}
F_{j,\delta}(x,z):=F_{\delta}(x,z)+\frac{\varepsilon_{j}}{q}(\lambda_\delta^2+\snr{z}^{2})^{\frac{q}{2}} \quad \mbox{and} \quad \mathcal{F}_{j,\delta}(w, B_R):=\int_{B_{R}}F_{j,\delta}(x,Dw)  \dx\;,
\end{flalign*}
where
\begin{flalign}\label{74}
\varepsilon_{j}:=\left(1+j+\nr{D\tilde{u}_{j}}^{2q}_{L^{q}(B_{R})}\right)^{-1} \Longrightarrow 
\frac{\varepsilon_{j}}{q}\int_{B_{R}}(\lambda^{2}+\snr{D\tilde{u}_{j}}^{2})^{\frac{q}{2}}  \dx\to0\;.
\end{flalign}
We moreover let 
\eqn{emme}
$$
m  := \frac {r}{r-2}\;.
$$
Using the definitions in \rif{mollificatori}, \rif{molli1} and \rif{molli2}, by convolution arguments (see Section \ref{arguments} below) we have that the integrand $F_{j,\delta}(\cdot)$ satisfies
\eqn{assF22}
$$
\left\{
\begin{array}{c}
\frac 1c[H_{\delta}(z)]^{\frac{2-\mu}{2}}+ \frac{\varepsilon_{j}}{q}[H_{\delta}(z)]^{\frac{q}{2}}\le F_{j,\delta}(x,z) \leq  cL[H_{\delta}(z)]^{\frac{q}{2}}+ cL[H_{\delta}(z)]^{\frac{2-\mu}{2}}\\ [8 pt]
\frac{1}{c}[H_{\delta}(z)]^{-\frac{\mu}{2}}+\frac{\eps_j}{c}[H_{\delta}(z)]^{\frac{q-2}{2}}\snr{\xi}^{2}\le \partial_{zz}F(x,z)\, \xi\cdot\xi\\ [8 pt]
\snr{\partial_{zz}F(x,z)}\le cL [H_{\delta}(z)]^{\frac{q-2}{2}}+  cL [H_{\delta}(z)]^{-\frac{\mu}{2}}\\ [6 pt]
\ \snr{\partial_{xz} F(x,z)}\le cL  h_{\delta}(x)[H_{\delta}(z)]^{\frac{q-1}{2}}+ cLh_{\delta}(x) [H_{\delta}(z)]^{\frac{1-\mu}{2}}  \\ [6 pt]
\snr{\partial_{xz}F_{\delta}(x,z)}\le cL\nr{h_{\delta}}_{L^{\infty}}[H_{\delta}(z)]^{\frac{q-1}{2}}\\[8pt]
\nr{h_{\delta}}_{L^{r}(B_R)}\le \nr{h}_{L^{r}(B_{R+\delta})}\;,
 \end{array}\right.
 $$
for every choice of $z, \xi\in \er^n$ and $x\in \overline{B_R}$, where $c  \equiv c   (n,\mu,q)\geq 1$. In the following we simply denote $\nr{h}_{L^{r}}\equiv \nr{h}_{L^{r}(B_{R+\delta})}$. 
By \rif{assF22}$_{1,2}$, Direct Methods and convexity we get that, for any $j,\delta$ as above, there exists a unique solution $u_{j,\delta}\in  \tilde{u}_{j}+W^{1,q}_{0}(B_{R})$ to the Dirichlet problem 
\eqn{pdmu}
$$
u_{j,\delta}\to \min_{w \in \tilde{u}_{j}+W^{1,q}_{0}(B_{R})} \mathcal{F}_{j,\delta}(w, B_R) \;.
$$
Thanks to \rif{assF22}$_1$ and as $u_{j,\delta}\in W^{1,q}_{0}(B_{R})$, the Euler-Lagrange equation of $ \mathcal{F}_{j,\delta}(\cdot)$ reads 
\begin{flalign}\label{88}
\int_{B_{R}}\partial_{z}F_{j,\delta}(x,Du_{j,\delta}) \cdot D\varphi  \dx=0 \ \ \mbox{for all} \ \varphi \in W^{1,q}_{0}(B_{R})\;.
\end{flalign}

\subsection{Step 2: Caccioppoli inequality}\label{step2} By \rif{assF22} and the smoothness implied by \rif{molli1}, classical regularity theory for non-degenerate equations with standard polynomial $q$-growth yields
\begin{flalign}\label{87}
u_{j,\delta}\in W^{1,\infty}_{\loc}(B_{R})\cap W^{2,2}_{\loc}(B_{R}) \quad \mbox{and}\quad
\partial F_{j,\delta}(\cdot, Du_{j,\delta}) \in W^{1,2}_{\loc}(B_R,\er^n)\;.
\end{flalign}
By virtue of \eqref{87}, we can differentiate equation \eqref{88} to obtain
\begin{flalign}\label{89}
\sum_{s=1}^{n}\int_{B_{R}}\left[\partial_{zz}F_{j,\delta}(x,Du_{j,\delta})DD_{s}u_{j,\delta}+\partial_{x_{s}z}F_{j,\delta}(x,Du_{j,\delta}) \right]\cdot D\varphi  \dx=0\;, \end{flalign}
for all $s \in \{1, \ldots, n\}$, which is valid whenever $\varphi\in W^{1,2}(B_{R})$ has compact support in $B_R$, again by \rif{87}. 
We select a cut-off function $\eta\in C^1_{\rm{c}}(B_{R})$ such that $0 \leq \eta \leq 1$. For every $s \in \{1, \ldots, n\}$, in \rif{89} we choose
\eqn{sceltatest}
$$\varphi\equiv \varphi_s
 :=\eta^{2}[H_{\delta}(Du_{j,\delta})]^{\gamma}D_{s}u_{j,\delta}\;.$$
This choice is again admissible by \rif{87} and it is
\begin{eqnarray*}
D\varphi_{s}&=&\eta^{2}\gamma [H_{\delta}(Du_{j,\delta})]^{\gamma-1}D_{s}u_{j,\delta}D(H_{\delta}(Du_{j,\delta}))+\eta^{2}[H_{\delta}(Du_{j,\delta})]^{\gamma}DD_{s}u_{j,\delta}\\ &&+2\eta [H_{\delta}(Du_{j,\delta})]^{\gamma}D_{s}u_{j,\delta}D\eta\;.
\end{eqnarray*}
We can rewrite \eqref{89} as
\eqn{tuttipezzi}
$$
0=\mbox{(I)}_{z}+\mbox{(II)}_{z}+\mbox{(III)}_{z}+\mbox{(I)}_{x}+\mbox{(II)}_{x}+\mbox{(III)}_{x}\;,
$$
where the terms indexed with $x$ denote the ones stemming from those in \eqref{89} containing $\partial_{xz}F$. 
Recalling that 
\eqn{derivataH}
$$
D \left[H_{\delta}(Du_{j,\delta})\right] = 2 \sum_{s=1}^{n} D_su_{j,\delta}DD_su_{j,\delta}\;,
$$ 
we have
\begin{eqnarray}
&&\nonumber \hspace{-3mm} \mbox{(I)}_{z}+\mbox{(II)}_{z}\\ &&:=\frac \gamma 2\int_{B_{R}}\eta^{2}\left( [H_{\delta}(Du_{j,\delta})]^{\gamma-1}\partial_{zz}F_{j,\delta}(x,Du_{j,\delta})D(H_{\delta}(Du_{j,\delta}))\cdot D(H_{\delta}(Du_{j,\delta}))  \right)\dx\nonumber \\
&&\qquad +\sum_{s=1}^{n}\int_{B_{R}}\eta^{2} [H_{\delta}(Du_{j,\delta})]^{\gamma}\partial_{zz}F_{j,\delta}(x,Du_{j,\delta})DD_{s}u_{j,\delta}\cdot DD_{s}u_{j,\delta} \dx\nonumber \\
&& \stackrel{\eqref{assF22}_2}{\geq} \frac{\gamma}{c}\int_{B_{R}}\eta^{2}[H_{\delta}(Du_{j,\delta})]^{\gamma-1-\frac{\mu}{2}}\snr{DH_{\delta}(Du_{j,\delta})}^{2}  \dx\nonumber \\
&& \qquad +\frac{1}{c}\int_{B_{R}}\eta^2[H_{\delta}(Du_{j,\delta})]^{\gamma-\frac{\mu}{2}}\snr{D^{2}u_{j,\delta}}^{2}  \dx\;.\label{sca1}
\end{eqnarray}
Young and H\"older inequality and $\eqref{assF22}_{3}$, instead give, for any $\sigma \in (0,1)$
\begin{eqnarray}\label{sca2}
\left|\mbox{(III)}_{z}\right|&\leq &2\sum_{s=1}^{n}\int_{B_{R}}\left|\eta [H_{\delta}(Du_{j,\delta})]^{\gamma}  D_{s}u_{j,\delta} \partial_{zz}F_{j,\delta}(x,Du_{j,\delta})DD_{s}u_{j,\delta}\cdot D\eta\right| \dx\nonumber \\
&\stackrel{\eqref{assF22}_3}{\leq} &cL \int_{B_{R}}\eta \snr{D\eta}\left([H_{\delta}(Du_{j,\delta})]^{\gamma+\frac{q-1}{2}}+[H_{\delta}(Du_{j,\delta})]^{\gamma+\frac{1-\mu}{2}}\right)\snr{D^{2}u_{j,\delta}}  \dx\nonumber \\
&\le &\sigma\int_{B_{R}}\eta^{2}[H_{\delta}(Du_{j,\delta})]^{\gamma-\frac{\mu}{2}}\snr{D^{2}u_{j,\delta}}^{2}  \dx\nonumber \\
&&\quad +\frac{cL ^2}{\sigma}\int_{B_{R}}\snr{D\eta}^{2}\left([H_{\delta}(Du_{j,\delta})]^{\gamma+q-1+\frac{\mu}{2}}+[H_{\delta}(Du_{j,\delta})]^{\gamma+\frac{2-\mu}{2}}\right)  \dx\nonumber \\
&\stackrel{2-\mu\leq q}{\leq}  &\sigma\int_{B_{R}}\eta^{2}[H_{\delta}(Du_{j,\delta})]^{\gamma-\frac{\mu}{2}}\snr{D^{2}u_{j,\delta}}^{2}  \dx\nonumber \\
 & & + \frac{cL ^2}{\sigma}  \int_{B_{R}}\snr{D\eta}^{2}\left(1+[H_{\delta}(Du_{j,\delta})]^{\gamma+q-1+\frac{\mu}{2}}\right)  \dx\nonumber \\
 &\leq  &\sigma\int_{B_{R}}\eta^{2}[H_{\delta}(Du_{j,\delta})]^{\gamma-\frac{\mu}{2}}\snr{D^{2}u_{j,\delta}}^{2}  \dx\nonumber \\
 &&+\frac{cL ^2 R^{\frac{2n}{r}}}{\sigma}\left(\int_{B_{R}}\snr{D\eta}^{2m}\left(1+[H_{\delta}(Du_{j,\delta})]^{m\left(\gamma+q-\frac{2-\mu}{2}\right)}\right) \dx\right)^{1/m}\:,
\end{eqnarray}
where we have used the elementary inequality $(1+t)^m\leq 2^{m-1}(1+t^m)$ for $t\geq 0$.  
Concerning the terms involving $\mbox{(I)}_{x}, \mbox{(II)}_{x}$ and $\mbox{(III)}_{x}$,
we have, by using H\"older and Young inequalities, and recalling the estimation for $\mbox{(III)}_{z}$
\begin{eqnarray}\label{95}
\nonumber \left|\mbox{(I)}_{x}\right|&\leq&\gamma \sum_{s=1}^{N}\int_{B_R}\left|\eta^{2}[H_{\delta}(Du_{j\delta})]^{\gamma-1}D_{s}u_{j\delta}\partial_{x_{s}z}F_{j\delta}(x,Du_{j\delta})\cdot D(H_{\delta}(Du_{j\delta})) \right| \dx\\ &
\stackrel{\eqref{assF22}_4}{\leq} &c L  \gamma \int_{B_{R}}\eta^{2}h_{\delta}(x)\left([H_{\delta}(Du_{j,\delta})]^{\gamma+\frac{q-2}{2}}+[H_{\delta}(Du_{j,\delta})]^{\gamma-\frac{\mu}{2}}\right)\snr{D(H_{\delta}(Du_{j,\delta}))}  \dx\nonumber \\
 &\le &\sigma\gamma \int_{B_{R}}\eta^{2}[H_{\delta}(Du_{j,\delta})]^{\gamma-1-\frac{\mu}{2}}\snr{D(H_{\delta}(Du_{j,\delta}))}^{2}  \dx\nonumber \\
 & & +  \frac{c\gamma L ^2}{\sigma} \int_{B_{R}}\eta^{2}[h_\delta(x)]^2\left([H_{\delta}(Du_{j,\delta})]^{\gamma+q-1+\frac{\mu}{2}}+[H_{\delta}(Du_{j,\delta})]^{\gamma+1-\frac{\mu}{2}}\right)  \dx\nonumber \\
&\stackrel{\eqref{assF22}_6}{\leq} &\sigma\gamma \int_{B_{R}}\eta^{2}[H_{\delta}(Du_{j,\delta})]^{\gamma-1-\frac{\mu}{2}}\snr{D(H_{\delta}(Du_{j,\delta}))}^{2}  \dx\nonumber \\
&&+ \frac{c\gamma L ^2}{\sigma}\nr{h}_{L^{r}}^{2}\left(\int_{B_{R}}\eta^{2m}\left(1+[H_{\delta}(Du_{j,\delta})]^{m\left(\gamma+q-\frac{2-\mu}{2}\right)}\right) \dx\right)^{1/m}\:.
\end{eqnarray}
Similarly, we have
\begin{eqnarray}
\left|\mbox{(II)}_{x}\right|&\leq &\sum_{s=1}^{n}\int_{B_{R}}\left|\eta^{2}[H_{\delta}(Du_{j,\delta})]^{\gamma}\partial_{x_{s}z}F_{j,\delta}(x,Du_{j,\delta})\cdot DD_{s}u_{j,\delta} \right| \dx\nonumber \\
&\stackrel{\eqref{assF22}_4}{\leq} &cL \int_{B_{R}}\eta^{2}h_{\delta}(x)\left([H_{\delta}(Du_{j,\delta})]^{\gamma+\frac{q-1}{2}}+[H_{\delta}(Du_{j,\delta})]^{\gamma+\frac{1-\mu}{2}}\right)\snr{D^{2}u_{j,\delta}}  \dx\nonumber  \\
&\le &\sigma \int_{B_{R}}\eta^{2}[H(Du_{j,\delta})]^{\gamma-\frac{\mu}{2}}\snr{D^{2}u_{j,\delta}}^{2}  \dx\nonumber \\
&&+ \frac{cL ^2}{\sigma}\int_{B_{R}}\eta^{2}
\left([H_{\delta}(Du_{j,\delta})]^{\gamma+q-1+\frac{\mu}{2}}+[H_{\delta}(Du_{j,\delta})]^{\gamma+1-\frac{\mu}{2}}\right)  \dx \nonumber \\
&\le &\sigma \int_{B_{R}}\eta^{2}[H(Du_{j,\delta})]^{\gamma-\frac{\mu}{2}}\snr{D^{2}u_{j,\delta}}^{2}  \dx\nonumber \\
&&+ \frac{cL ^2}{\sigma}\nr{h}^{2}_{L^{r}}\left(\int_{B_{R}}\eta^{2m}\left(1+[H_{\delta}(Du_{j,\delta})]^{m\left(\gamma+q-\frac{2-\mu}{2}\right)}\right)  \dx \right)^{1/m}\,,\label{94}
\end{eqnarray}
and using that $(2-\mu)/2 \leq q/2\leq q-1+\mu/2$, we have
\begin{eqnarray}\label{96}
\left|\mbox{(III)}_{x}\right|&\leq &2\sum_{s=1}^{n}\int_{B_{R}}\left|\eta [H_{\delta}(Du_{j,\delta})]^{\gamma}D_su_{j,\delta} \partial_{xz}F(x,Du_{j,\delta})\cdot D\eta  \right| \dx\nonumber \\
&\stackrel{\eqref{assF22}_4}{\leq} &cL  \int_{B_{R}}\eta \snr{D\eta}h_{\delta}(x)\left([H_{\delta}(Du_{j,\delta})]^{\gamma+\frac{q}{2}}
+[H_{\delta}(Du_{j,\delta})]^{\gamma+\frac{2-\mu}{2}}\right)  \dx\nonumber \\
&\le &cL  \int_{B_{R}}\left(\eta^2+ \snr{D\eta}^2\right)h_{\delta}(x)\left(1+[H_{\delta}(Du_{j,\delta})]^{\gamma+q-1+\frac{\mu}{2}}
\right)  \dx\nonumber \\
&\le &cL^2  \nr{h}^{2}_{L^{r}}\left(\int_{B_{R}}(\eta^{2m}+\snr{D\eta}^{2m})\left(1+[H_{\delta}(Du_{j,\delta})]^{m\left(\gamma+q-\frac{2-\mu}{2}\right)}\right)  \dx\right)^{1/m}\;.
\end{eqnarray}
We have used that $L\geq 1$. Choosing $\sigma\equiv \sigma (n,\mu,q)$ sufficiently small in order to reabsorb terms, and merging estimates \eqref{sca1}-\eqref{96} with \rif{tuttipezzi}, we obtain
\begin{eqnarray}
&&\label{9800}\qquad   \gamma\int_{B_{R}}\eta^{2}[H_{\delta}(Du_{j,\delta})]^{\gamma-1-\frac{\mu}{2}}\snr{DH_{\delta}(Du_{j,\delta})}^{2}  \dx +\int_{B_{R}}\eta^2[H_{\delta}(Du_{j,\delta})]^{\gamma-\frac{\mu}{2}}\snr{D^{2}u_{j,\delta}}^{2}  \dx \\
&&\hspace{-1mm} \le c\left[(1+\nr{h}_{L^{r}})L\right]^2(1+\gamma)\left(\int_{B_{R}}(\eta^{2m}+\snr{D\eta}^{2m})\left(1+[H_{\delta}(Du_{j,\delta})]^{m\left(\gamma+q-\frac{2-\mu}{2}\right)}\right)  \dx\right)^{\frac{1}{m}}\nonumber
\end{eqnarray}
for $c \equiv c(n,\mu,q)$. By Sobolev embedding theorem, recalling that $\mu <2$ by \eqref{pq}, we have (using the elementary inequality $1 + t^{2^*}\leq (1+ t)^{2^*}$ for $t \geq 0$ and \rif{derivataH})
\begin{eqnarray}
\nonumber &&\hspace{-7mm} \left(\int_{B_{R}}\eta^{2^{*}}\left(1+[H_{\delta}(Du_{j,\delta})]^{\left(\gamma+\frac{2-\mu}{2}\right)\frac{2^{*}}{2}}\right)   \dx \right)^{2/2^{*}} \\&& \nonumber \leq  \left(\int_{B_{R}}\eta^{2^{*}}(1+[H_{\delta}(Du_{j,\delta})]^{\frac \gamma 2 +\frac{2-\mu}{4}})^{2^{*}}   \dx \right)^{2/2^{*}}\\
\nonumber &&   \le c\tilde R\int_{B_{R}}\left |D\left(\eta\left(1+[H_{\delta}(Du_{j,\delta})]^{\frac \gamma 2 +\frac{2-\mu}{4}}\right)\right) \right |^{2}  \dx\nonumber \\
\nonumber &&   \le c\tilde R\int_{B_{R}}\snr{D\eta}^{2}\left(1+[H_{\delta}(Du_{j,\delta})]^{\gamma+\frac{2-\mu}{2}}\right)  \dx +c\tilde R\int_{B_{R}}\eta^{2}|D[H_{\delta}(Du_{j,\delta})]^{\frac{\gamma}{2}+\frac{2-\mu}{4}}|^{2}  \dx \nonumber \\
&&    \le  c\tilde R\left(\int_{B_{R}}\snr{D\eta}^{2m}\left(1+[H_{\delta}(Du_{j,\delta})]^{m\left(\gamma+q-\frac{2-\mu}{2}\right)}\right)  \dx\right)^{1/m}\nonumber \\
&& \quad +  c\gamma^2\tilde R\int_{B_{R}}\eta^{2}[H_{\delta}(Du_{j,\delta})]^{\gamma-1-\frac{\mu}{2}}\snr{DH_{\delta}(Du_{j,\delta})}^{2}  \dx \nonumber \\
&& \quad \quad 
+c \tilde R \int_{B_{R}}\eta^2[H_{\delta}(Du_{j,\delta})]^{\gamma-\frac{\mu}{2}}\snr{D^{2}u_{j,\delta}}^{2}  \dx
\label{emby}
\end{eqnarray}
for $c\equiv c(n,\mu,q)$. Here, we are denoting 
\eqn{immergi}
$$
2^{*} = \left\{ \begin{array}{ccc}
\frac{2n}{n-2} &\mbox{if} & n>2\\[10 pt]
\mbox{any number strictly larger than $4m$} &\mbox{if} & n=2
\end{array}
\right.
$$
and 
\eqn{tilder}
$$
\tilde R := R^{\frac{2n}{2^*}-n+2} \Longrightarrow\  \mbox{$\tilde R =1$ if $n>2$} \;.
$$
Using \rif{9800} to estimate the last two terms in display \rif{emby}, we conclude with the following basic reverse H\"older inequality:
\begin{eqnarray}\label{0}
&&\quad \left(\int_{B_{R}}\eta^{2^{*}} \left(1+[H_{\delta}(Du_{j,\delta})]^{\left(\gamma+\frac{2-\mu}{2}\right)\frac{2^{*}}{2}}\right)  \dx \right)^{\frac{2}{2^{*}}} \\
&&   \le c[(1+\nr{h}_{L^{r}})L(1+\gamma)]^2\tilde R \left(\int_{B_{R}}(\eta^{2m}+\snr{D\eta}^{2m})\left(1+[H_{\delta}(Du_{j,\delta})]^{m\left(\gamma+q-\frac{2-\mu}{2}\right)}\right)  \dx\right)^{\frac{1}{m}}\nonumber 
\end{eqnarray}
for $c \equiv c(n,\mu,q)$. 

\subsection{Step  3: Modified Moser's iteration} \label{moser1} We inductively define the exponents
$$
\gamma_{1}:=0, \quad \gamma_{k +1}:=\frac{1}{m}\left[\left(\gamma_{k}+\frac{2-\mu}{2}\right)\frac{2^{*}}{2}-\frac{2-\mu}{2}\right], \quad \alpha_{k }:=m\gamma_{k}+\frac{2-\mu}{2}\;,
$$
for every integer $k\geq 1$, where $m$ has been introduced in \rif{emme}. It follows that
\begin{flalign}\label{1}
\alpha_{k +1}=\left(\gamma_{k }+\frac{2-\mu}{2}\right)\frac{2^{*}}{2}=\chi \alpha_{k }+\tau\qquad \mbox{for every $k\geq 1$}\;,
\end{flalign}
where it is
\begin{flalign}\label{7}
\chi:=\frac{2^{*}}{2m}\stackrel{r>n}{>}1 \ \ \mbox{and} \ \ \tau:=\frac{2^{*}\alpha_1}{r}=\frac{2^{*}(2-\mu)}{2r}\stackrel{2>\mu}{>}0\;.
\end{flalign}
As a consequence of \eqref{1}, by induction we have that the identities 
\eqn{2}
$$
\alpha_{k +1}=\chi^{k }\alpha_{1}+\tau\sum_{i=0}^{k -1}\chi^{i} \quad \mbox{and therefore} \quad  \gamma_{k +1}= \frac{\alpha_1}{m}\left(\chi^{k }-1\right) +
\frac{\tau}{m}\sum_{i=0}^{k -1}\chi^{i}
$$
hold for every integer $k\geq 1$. 
Being $\chi>1$, it is 
$
\alpha_{k +1}> \alpha_{k }$ for every $k \in \mathbb{N}$.
For later use, we record the elementary estimation 
\eqn{elegamma}
$$
\gamma_{k +1}\leq \frac{2\alpha_1}{\chi-1}\chi^{k+1}= \frac{2-\mu}{\chi-1}\chi^{k+1}\qquad \mbox{for every $k\geq 1$}\;.
$$
In the following all the balls considered will be concentric to $B_R$. We abbreviate as 
\eqn{espressione}
$$
M_{j,\delta}(\mathcal{l}):=\nr{H_{\delta}(Du_{j,\delta})}_{L^{\infty}(B_{\mathcal{l}})} \quad  \mbox{for $\mathcal{l}\in (0,R)$}\;.
$$
By \rif{87} this function is bounded on every interval $[\varrho, R_*]$, whenever $\varrho <R_*< R$. 
For $0<\rr\le \tau_1 < \tau_2<  R$, we consider a sequence $\{B_{\varrho_k}\}$ of shirking balls, where $\rr_{k }:=\tau_1+(\tau_2-\tau_1)2^{-k +1}$. Notice that $\{\varrho_k\}$ is a decreasing sequence such that $\varrho_1=\tau_2$ and $\varrho_k\to \tau_1$; therefore it is $\cap_k B_{\varrho_k} =B_{\tau_1}$ and $B_{\varrho_1}= B_{\tau_2}$. Accordingly, we fix corresponding cut-off functions $\eta_{k }\in C^1_{\rm{c}}(B_{R})$ with
$$
 \mathds{1}_{B_{\rr_{k +1}}}\le \eta\le  \mathds{1}_{B_{\rr_{k }}} \ \ \mbox{and} \ \ \snr{D\eta_{k }}\lesssim \frac{1}{(\rr_{k }-\rr_{k +1})}\approx \frac{2^{k }}{\tau_2-\tau_1}\;.
$$
We choose $\eta \equiv \eta_k$ in \eqref{0} and manipulate as to obtain, with the above notation 
\begin{eqnarray}\label{3}
&& \quad \int_{B_{\rr_{k +1}}}\left(1+[H_{\delta}(Du_{j,\delta})]^{\alpha_{k +1}}\right)  \dx\\
 &&\qquad\le  \left[ \frac{c_h 2^{k }(1+ \gamma_{k })}{\tau_2-\tau_1}\right]^{2^{*}} \left(1+
[M_{j,\delta}(\tau_2)]^{\frac{2^{*}\sigma}{2}}\right)\left(\int_{B_{\rr_{k }}}\left(1+[H_{\delta}(Du_{j,\delta})]^{\alpha_{k }}\right)  \dx\right)^{\chi} \nonumber
\end{eqnarray}
where (recall that $q \geq 2-\mu$) it is
\eqn{ilsigma}
$$\sigma:=q-\frac{2-\mu}{2}-\frac{2-\mu}{2m}= q-\alpha_1-\frac{\alpha_1}{m}\geq 0\;.$$ 
As for $c_h$, with this symbol we denote a fixed number of the form 
\eqn{ch}
$$c_h \equiv \tilde c(n,\mu,q)\left(1+\nr{h}_{L^{r}(B_{R+\delta})}\right)L\;.$$ 
In the following we shall emphasize, in describing the dependence of the various constants on $\mu$ and $\chi$, the asymptotic behaviour for $\mu\to 2$ and $\chi \to 1$. For $k \in \mathbb{N}$, we define
\eqn{gli}
$$
A_{k }:=\left(\int_{B_{\rr_{k }}}\left(1+[H_{\delta}(Du_{j,\delta})]^{\alpha_{k }}\right)  \dx\right)^{1/\alpha_{k }}\,, \qquad \tilde \gamma_{k} := 1 +\gamma_{k} 
$$
thus \eqref{3} becomes
$$
A_{k +1}\le \left( \frac{c_h 2^{k }\tilde \gamma_{k }}{\tau_2-\tau_1}\right)
^{\frac{2^{*}}{\alpha_{k +1}}}\left(1+[M_{j,\delta}(\tau_2)]^{\frac{2^{*}\sigma}{2}}\right)^{\frac{1}{\alpha_{k +1}}}A_{k }^{\frac{\chi\alpha_{k }}{\alpha_{k +1}}}\;.
$$
Iterating the last inequality gives  
$$
A_{k +1}\le \prod_{i=0}^{k -1}\left(\frac{c_h2^{k -i}\tilde \gamma_{k -i}}{\tau_2-\tau_1}\right)^{\frac{2^{*}\chi^{i}}{\alpha_{k +1}}}\left(1+[M_{j,\delta}(\tau_2)]^{\frac{2^{*}\sigma}{2}}\right)^{\frac{1}{\alpha_{k +1}}\sum_{i=0}^{k -1}\chi^{i}}A_{1}^{\frac{\chi^{k }\alpha_{1}}{\alpha_{k +1}}}\;,
$$
for every $k\geq 1$, 
and, noticing that 
\eqn{8}
$$
\lim_{k \to \infty}\frac{1}{\alpha_{k +1}}\sum_{i=0}^{k -1}\chi^{i}\stackrel{\rif{2}}{=}\frac{1}{(\chi-1)\alpha_{1}+\tau}  \leq  \frac{2}{(\chi-1)\alpha_1}=  \frac{2}{(2-\mu)(\chi-1)}\;,
$$
we can further bound as 
\eqn{77}
$$
A_{k +1}\leq 4^{\frac{1}{(2-\mu)(\chi-1)}}\prod_{i=0}^{k -1}\left(\frac{c_h2^{k -i}\tilde \gamma_{k -i}}{\tau_2-\tau_1}\right)^{\frac{2^{*}\chi^{i}}{\alpha_{k +1}}}\left(1+[M_{j,\delta}(\tau_2)]^{\theta}\right)A_{1}^{\frac{\chi^{k }\alpha_{1}}{\alpha_{k +1}}}\;.
$$
Here we have denoted 
\eqn{iltheta}
$$
\theta:=\frac{2^{*}\sigma}{2\left[(\chi-1)\alpha_{1}+\tau\right]}
=\frac{\frac{\sigma}{\alpha_1}\frac{2^*}{2}}{\left(\frac 1m + \frac 2r\right)\frac{2^*}{2}-1}\stackrel{\rif{7}}{=} 
\frac{\chi m\sigma}{(\chi-1)\alpha_{1}+\tau}
\;.
$$
Using \rif{2} we next compute
\eqn{euna00}
$$
\lim_{k \to \infty}\frac{\chi^{k }\alpha_{1}}{\alpha_{k +1}}=\frac{(\chi-1)\alpha_{1}}{(\chi-1)\alpha_{1}+\tau}\;,
$$
\eqn{euna}
$$
\frac{2^{*}}{\alpha_{k +1}}\sum_{i=0}^{k -1}\chi^{i}(k -i) \leq \frac{2^{*}}{\alpha_1}\sum_{i=0}^{k -1}\frac{k -i}{\chi^{k -i}}\leq  \frac{c(\chi)}{(2-\mu)(\chi-1)^2}
$$
and therefore (recalling that $\tilde \gamma_1=1$) we also have, using \rif{2} and \rif{elegamma}
 \begin{eqnarray*}
 \prod_{i=0}^{k -1}\tilde{\gamma}_{k -i}^{\frac{2^{*}\chi^{i}}{\alpha_{k +1}}} &=& \prod_{i=0}^{k -2}\tilde{\gamma}_{k -i}^{\frac{2^{*}\chi^{i}}{\alpha_{k +1}}}= \exp \left(\frac{2^*}{\alpha_{k +1}}\sum_{i=0}^{k -2}
\chi^i  \log \tilde \gamma_{k-i}\right)\\
& \leq &
 \exp \left[\frac{2^*}{2-\mu}\sum_{i=0}^{k -2}
\frac{\log (e+ \chi^{k-i})}{\chi^{k-i}} + \frac{2^*}{2-\mu}\log \left(e+ \frac{2-\mu}{\chi-1}\right)\sum_{i=0}^{k -2}
\frac{1}{\chi^{k-i}}\right] \\
 &\leq &
 \exp \left[\frac{c(\chi)}{2-\mu}\sum_{i=0}^{k -2}
\frac{k-i}{\chi^{k-i}} + \frac{c(\chi, \mu)}{(2-\mu)(\chi-1)}\sum_{i=0}^{k -2}
\frac{1}{\chi^{k-i}}\right]\\ & \leq &
 \exp \left[\frac{c(\mu, \chi)}{(2-\mu)(\chi-1)^2} \right] 
 \end{eqnarray*}
 so that we finally conclude with
 \begin{eqnarray}
\nonumber\prod_{i=0}^{k -1}\left(\frac{c_h2^{k -i}\tilde \gamma_{k -i}}{\tau_2-\tau_1}\right)^{\frac{2^{*}\chi^{i}}{\alpha_{k +1}}}&
 \le & 2^{\frac{2^{*}}{\alpha_{k +1}}\sum_{i=0}^{k -1}\chi^{i}(k -i)}\prod_{i=0}^{k -1}\tilde{\gamma}_{k -i}^{\frac{2^{*}\chi^{i}}{\alpha_{k +1}}}\left(\frac{c_h}{\tau_2-\tau_1}\right)^{\frac{2^{*}}{\alpha_{k +1}}\sum_{i=0}^{k -1}\chi^{i}}
\\
& \stackleq{8} & 4^{\frac{c (\mu, \chi)}{(2-\mu)(\chi-1)^2}}\left(\frac{c_h}{\tau_2-\tau_1}\right)^{\frac{2\chi m}{(\chi-1)\alpha_{1}+\tau}}\;.\label{plug}
\end{eqnarray}
Plugging \rif{plug} in \rif{77}, letting $k  \to \infty$ there, and taking \rif{euna00}-\rif{euna} into account, we find, after an elementary estimation 
\eqn{77bis}
$$
M_{j,\delta}(\tau_1)\le 4^{\frac{c (\mu, \chi)}{(2-\mu)(\chi-1)^2}}\left(\frac{c_h}{\tau_2-\tau_1}\right)^{\frac{2\chi m}{(\chi-1)\alpha_{1}+\tau}}\left(1+[M_{j,\delta}(\tau_2)]^{\theta}\right)A_{1}^{\frac{(\chi-1)\alpha_{1}}{(\chi-1)\alpha_{1}+\tau}}\;, 
$$
where the constant $c (\mu, \chi)$ remains bounded for $\chi\to 1$ and $\mu \to 2$ and $c_h$ has been defined in \rif{ch}.   
Now we concentrate on the case $n>2$; using the expressions in \rif{ilsigma} and \rif{iltheta}, we notice that $\theta <1$ if and only if the last condition assumed in \rif{pq} is satisfied. Therefore we can apply Young  inequality with conjugate exponents $1/\theta$ and $1/(1-\theta)$ in \rif{77bis}; this yields
$$
\nonumber M_{j,\delta}(\tau_1)\le  \frac{1}{2}M_{j,\delta}(\tau_2)+ c\left(\frac{L+L\|h\|_{L^{r}(B_{R+\delta})}}{\tau_2-\tau_1}\right)^{2\kk_1}\left(1+A_1\right)^{2\alpha_1\kk_{2}}\;,
$$
where $c,\kk_1, \kk_2\equiv c,\kk_1, \kk_2(n,\mu,q,r)$, and we have restored the full notation from \rif{ch}.  
Recalling that $\varrho \leq \tau_1 < \tau_2 < R$, and the expressions in \rif{espressione}, and \rif{gli}, Lemma \ref{l0} applied to $\mathcal{Z}(\mathcal{l})\equiv \nonumber M_{j,\delta}(\mathcal{l})$ (which is bounded on every interval $[\varrho, R_*]$, $R_*<R$, by \rif{87}) yields 
\eqn{pre98bis}
$$
\nr{H_{\delta}(Du_{j,\delta})}_{L^{\infty}(B_{\rr})}\le c\left(\frac{L+L\|h\|_{L^{r}(B_{R+\delta})}}{R-\varrho}\right)^{2\kk_1}\left(1+\int_{B_{R}}[H_{\delta}(Du_{j,\delta})]^{\frac{2-\mu}{2}} \dx\right)^{2\kappa_2}\;.
$$
Eventually, using \rif{assF22}$_{1}$ in the last estimate we obtain the desired a priori bound, i.e., 
\begin{flalign}\label{98bis}
\nr{H_{\delta}(Du_{j,\delta})}_{L^{\infty}(B_{\rr})}\le c\left(\frac{L+L\|h\|_{L^{r}(B_{R+\delta})}}{R-\rr}\right)^{2\kappa_1}\left[1+\mathcal{F}_{j,\delta}(u_{j,\delta},B_{R})\right]^{2\kappa_2}\;,
\end{flalign}
where $c,\kk_1, \kk_2\equiv c,\kk_1, \kk_2(n,\mu,q,r)$, and which is now established for the case $n>2$. 
It remains to treat the case $n=2$. For this, recalling \rif{iltheta}, we notice that
$$
\lim_{2^* \to \infty} \theta = \frac{\sigma}{\alpha_1(1/m+2/r)} = \frac{\sigma}{\alpha_1}
$$
while, using \rif{ilsigma}, we find
$$
 \frac{\sigma}{\alpha_1}< 1 \Longleftrightarrow  \frac{q}{2-\mu}<1+\frac{r-2}{2r}\;,
$$
which is the bound assumed in \rif{pq} for $n=2$. We can therefore take $2^*$ large enough (recall the definition in \rif{immergi}) in order to have $\theta<1$ once again and proceed as in the case $n>2$ after \rif{iltheta}. We again conclude with \rif{98bis} for different values of the exponents $\kk_1, \kk_2$. 

\begin{remark}\label{esp1}   \emph{A careful check of the above proof shows that the constant $c$ appearing in \rif{98bis}
blows-up when $r \to n$ (this implies $\chi \to 1$), $\mu\to 2$. This constant remains instead stable when $r \to \infty$. The exponents $\kappa_1$ and $\kappa_2$ in \rif{98bis} turn out to be
\eqn{iduekappa}
$$
\kk_1= \frac{\chi m}{(\chi-1)\alpha_1+\tau-m\chi\sigma} \quad \mbox{and} \quad
\kk_2= \frac{\chi-1}{2\left[(\chi-1)\alpha_1+\tau-m\chi\sigma\right]}\;,
$$
respectively. 
Compare with Remark \ref{esp2} below. As a matter of fact the above proof perfectly works in the case $r = \infty$, that means, with the above notation, $m=1$. This is the case of Lipschitz continuous coefficients, revisited in Section \ref{moser2} below.
}
\end{remark}

\subsection{Step 4: Passage to the limit and conclusion.}\label{chiudi} 
We consider a sequence of numbers $\{\delta_k\}$ such that $\delta_k \to 0$ and we take $\delta\equiv  \delta_k$ in \rif{molli1}. In fact, we keep using the notation $\delta\equiv  \delta_k$. Moreover, we denote $\delta \to 0$ for $k \to \infty$. We shall several times pass to subsequences, still denoted by $\delta$. We now fix an arbitrary index $j \in \en$ and a concentric ball $B_{\varrho}\Subset B_R$ as in the statement of Theorem \ref{main3}. We have
\begin{eqnarray}
\nonumber  \frac{\varepsilon_{j}}{q}\int_{B_{R}}(\lambda^{2}_{\delta}+\snr{Du_{j,\delta}}^{2})^{\frac{q}{2}}  \dx &\stackrel{\eqref{assF22}_1} {\leq}&  \mathcal{F}_{j,\delta}(u_{j,\delta}, B_R)\\ \nonumber &\stackleq{pdmu}& \mathcal{F}_{j,\delta}(\tilde{u}_{j},B_R)\\ \nonumber &\stackleq{74} &
\int_{B_R}F_{\delta}(x, D\tilde  u_j) \dx+\oo(j)\\ \nonumber  &= & \mathcal{F}(\tilde u_j,B_{R})+\oo(j)+\int_{B_{R}}\left[F_{\delta}(x,D\tilde u_{j})-F(x,D\tilde u_{j}) \right] \dx \\ &\stackrel{\rif{73}}{=} & \mathcal{F}(u,B_{R})+\oo(j)+\int_{B_{R}}\left[F_{\delta}(x,D\tilde u_{j})-F(x,D\tilde u_{j}) \right] \dx \;.\label{upto}
\end{eqnarray}
By \rif{81} - recall that here $j$ is fixed and $D\tilde u_{j}\in W^{1,q}(B_R)$ - the last integral in the above display goes to zero as $\delta \to 0$, i.e.
\eqn{piccolino}
$$
\int_{B_{R}}\left[F_{\delta}(x,D\tilde u_{j})-F(x,D\tilde u_{j}) \right] \dx=: \oo_j(\delta)
$$
and $\oo_j(\delta)\to 0$ we $\delta \to 0$. 
We conclude that the sequence $\{Du_{j,\delta}\}_{\delta}$ is bounded in $L^q(B_R)$. 
Therefore, up to a not relabelled subsequence (depending on the chosen index $j\in \en$), we find $u_{j}\in \tilde{u}_{j}+ W^{1,q}_0(B_{R})$ such that $
 u_{j,\delta}\rightharpoonup u_{j}$  weakly in $W^{1,q}(B_{R})$ as $\delta\to 0$. 
By \rif{98bis}, \rif{upto} and \rif{piccolino} we now have 
\begin{flalign}\label{98bis2}
\nr{Du_{j,\delta}}_{L^{\infty}(B_{\rr})}\le c\left(\frac{L+L\|h\|_{L^{r}(B_{R+\delta})}}{R-\rr}\right)^{\kappa_1}\left[1+\mathcal{F}(u,B_{R})+ \oo(j)+ \oo_j(\delta)\right]^{\kappa_2}\;,
\end{flalign}
for $c\equiv c(n,\mu,q,r)$. This implies that, up to a not relabelled subsequence, we have 
$
u_{j,\delta}\stackrel{*}{\rightharpoonup}u_{j}$ in $W^{1,\infty}(B_{\rr},\RN)$ as $\delta \to 0$. 
By $\mbox{weak}^{*}$-lower semicontinuity, letting $\delta\to 0$ in \eqref{98bis2} we find
\begin{flalign}\label{98bis3}
\nr{Du_{j}}_{L^{\infty}(B_{\rr})}\le c\left(\frac{L+L\|h\|_{L^{r}(B_{R})}}{R-\rr}\right)^{\kappa_1}\left[1+\mathcal{F}(u,B_{R})+\oo(j)\right]^{\kappa_2}\;,
\end{flalign}
with $c,\kk_1, \kk_2\equiv c,\kk_1, \kk_2(n,\mu,q,r)$. By \rif{81} and \rif{98bis2}  we have
$$
\lim_{\delta\to 0}\int_{B_{\varrho}}\left[F_{\delta}(x,Du_{j,\delta})-F(x,Du_{j,\delta}) \right] \dx =0\;.
$$
As by lower semicontinuity we also have 
$$
\mathcal{F}(u_{j},B_{\varrho})\leq \liminf_{\delta\to 0} \mathcal{F}(u_{j,\delta},B_{\varrho})\;,
$$
we conclude with 
\begin{eqnarray*}
\mathcal{F}(u_{j},B_{\varrho})&\leq & \liminf_{\delta\to 0}\int_{B_{\varrho}} F_{\delta}(x,Du_{j,\delta})  \dx\\
&\leq &\limsup_{\delta\to 0} \mathcal{F}_{j,\delta}(u_{j,\delta},B_{R})
\stackrel{\rif{upto},\rif{piccolino}}{\leq} \mathcal{F}(\tilde{u}_{j},B_{R})+\oo(j)\;.
\end{eqnarray*}
Letting $\rr\to R$ in the above inequality finally gives
\begin{flalign}\label{129}
 \mathcal{F}(u_j,B_{R}) \le  \mathcal{F}(\tilde u_j,B_{R}) +\oo(j)
 \stackrel{\rif{73}}{=} \mathcal{F}(u,B_{R}) +\oo(j)\;.
\end{flalign}
By \rif{assF2}$_3$ and this last estimate it follows that the sequence $\{\bar{F}(\snr{Du_{j}})\}$ is bounded in $L^1(B_R)$. Now, \rif{assF3}$_2$ and Dunford-Pettis criterion imply that, up to not relabelled subsequences, there exists $v \in u +W^{1,1}_0(B_R)$ such that
\begin{flalign}\label{123}
u_{j}\rightharpoonup v \ \mbox{in} \ W^{1,1}(B_{R}) \ \mbox{and} \ \ u_{j}\stackrel{*}{\rightharpoonup}v \ \mbox{in} \ W^{1,\infty}(B_{\rr})
\;.
\end{flalign}
Next, \rif{129} and  \rif{123} imply $ \mathcal{F}(v,B_{R})\leq  \mathcal{F}(u,B_{R})$ via lower semicontinuity. On the other hand, as $u-v \in W^{1,1}_0(B_R)$, the minimality of $u$ yields 
$ \mathcal{F}(u,B_{R})\leq  \mathcal{F}(v,B_{R})$ so that $ \mathcal{F}(v,B_{R})=\mathcal{F}(u,B_{R})$. In turn, the strict convexity of the functional $\mathcal F$ implies $u=v$. 
Finally, using this last fact and $\eqref{123}_{2}$, letting $j \to \infty$ in \rif{98bis3}, lower semicontinuity provide us with \rif{main-ine-lip}. This holds when $\nu=1$; the general case $\nu\not=1$ can be achieved by scaling as said at the very beginning of Step 1. The proof of Theorem \ref{main3} is finally complete up to some clarifications contained in the next final step. 

\subsection{Step 5: Arguments for \rif{assF22}}\label{arguments}
The arguments leading to the precise statement \rif{assF22} are not easy to find in the literature. We therefore report the needed proofs also because we think that they are useful elsewhere; for instance, when dealing with higher gradient regularity. For this, it is sufficient to consider an integrand $G \in C^{2}_{\loc}(\er^n\setminus \{0\})\cap C^{1}_{\loc}(\er^n)$ satisfying
\eqn{assG}
$$
\left\{
\begin{array}{c}
\nu(\lambda^{2}+\snr{z}^{2})^{\frac{q}{2}}\le G(z)\le L(\lambda^{2}+\snr{z}^{2})^{\frac{q}{2}}+  L(\lambda^{2}+\snr{z}^{2})^{\frac{\gamma}{2}}\\ [6 pt]
\nu (\lambda^{2}+\snr{z}^{2})^{\frac{\gamma-2}{2}}\le \partial^2 G(z)\, \xi\cdot\xi\\ [6 pt]
\snr{\partial^2 G(z)}\le L (\lambda^{2}+\snr{z}^{2})^{\frac{q-2}{2}}+L (\lambda^{2}+\snr{z}^{2})^{\frac{\gamma-2}{2}}\;,
 \end{array}\right.
 $$
for every $z, \xi \in \er^n$, $|z|\not=0$, where $q \geq \max\{1,\gamma\}$, $\gamma > 0$, $0 < \nu \leq  1 \leq L $ are fixed constants. We then consider, for $\delta \in (0,1)$
$$
G_{\delta}(z):=\mint_{B_{1}}G(z+\delta  y) \phi(y) \dy \quad \mbox{and}\quad \lambda_\delta := \lambda+\delta
\;,
$$ 
where $\phi(\cdot)$ is as in \rif{mollificatori} and $\delta \in (0,1)$. The newly defined function $G_{\delta}(\cdot)$ satisfies 
\eqn{assGdelta}
$$
\left\{
\begin{array}{c}
\frac { \nu }c(\lambda^{2}_{\delta}+\snr{z}^{2})^{\frac{q}{2}}\le G_{\delta}(z)\le cL(\lambda^{2}_{\delta}+\snr{z}^{2})^{\frac{q}{2}}+cL(\lambda^{2}_{\delta}+\snr{z}^{2})^{\frac{\gamma}{2}}\\ [6 pt]
\frac \nu c (\lambda^{2}_{\delta}+\snr{z}^{2})^{\frac{\gamma-2}{2}}\le \partial^2 G_{\delta}(z)\, \xi\cdot\xi\\ [6 pt]
\snr{\partial^2 G_{\delta}(z)}\le cL(\lambda^{2}_{\delta}+\snr{z}^{2})^{\frac{q-2}{2}}+cL(\lambda^{2}_{\delta}+\snr{z}^{2})^{\frac{\gamma-2}{2}}\;,
 \end{array}\right.
 $$
for every $z, \xi \in \er^n$, where $c\equiv c (n,\gamma,q)$. Observe that once we have proved \rif{assGdelta}, the validity of \rif{assF22} easily follows applying the arguments here to the integrand $G(z):= F(x+\delta\bar{y}, z)$ for every $x\in \overline{B_R}$ and $y \in B_1(0)$, and using basic properties of convolutions.

In order to prove \rif{assGdelta}, we first observe that the proof of \rif{assGdelta}$_1$ follows exactly as in \cite[Lemma 3.1, Step 1]{ELM2}; in particular, the upper bound in \rif{assGdelta}$_1$ is trivial. We also notice that in \cite[Lemma 3.1]{ELM2} a proof of \rif{assGdelta}$_{2,3}$ is provided for different values of $\gamma$ and, more importantly, with a dependence of the constant $c$ on $\sigma$, as far as \rif{assGdelta}$_{3}$ is concerned. This is not the case in \rif{assGdelta} and different arguments from those of \cite{ELM2} are required. We divide the proof of \rif{assGdelta} in two cases. 
 
\emph{Case 1: $0<\gamma<2$.} First, we consider the case $0<\snr{z}\le \sqrt{32}\delta$ and $\lambda\le \sqrt{16}\delta$, then, using the definition of $G_{\delta}(\cdot)$, we have, recalling that $q \geq \gamma$, integrating by parts and using \rif{assG}$_1$
\begin{eqnarray*}
\snr{\partial^2 G(z)}&= &  \delta^{-2}\left|\mint_{B_{1}}G(z+\delta y)\partial^2\phi(y )  \dy  \right|\nonumber \\
&\le &cL\delta^{-2}\mint_{B_{1}}(\lambda^{2}+\snr{z+\delta y}^{2})^{\frac{q}{2}} \d y+cL\delta^{-2}\mint_{B_{1}}(\lambda^{2}+\snr{z+\delta y}^{2})^{\frac{\gamma}{2}}  \dy \nonumber \\
&\le &cL\delta^{-2}(\lambda^{2}+\delta^2+ \snr{z}^{2})^{\frac{\gamma}{2}}\le cL\delta^{\gamma-2}\le c(\gamma,q)L(\lambda_{\delta}^{2}+\snr{z}^{2})^{\frac{\gamma-2}{2}}
\end{eqnarray*}
so that \rif{assGdelta}$_3$ follows in this case. If $\lambda>\sqrt{16}\delta$, we estimate
\begin{eqnarray}
\nonumber \snr{\partial^2 G(z)} &\le & cL\mint_{B_{1}}(\lambda^{2}+\snr{z+\delta y }^{2})^{\frac{q-2}{2}} \phi(y) \dy \\ && \    +cL\mint_{B_{1}}(\lambda^{2}+\snr{z+\delta y }^{2})^{\frac{\gamma-2}{2}} \phi(y) \dy =: I_1 + I_2\;.\label{sti1}
\end{eqnarray}
Notice that there is a potential problem with the convergence of the last two integrals (of the first only in the case $q<2$) when $\lambda=0$; the two integrals are anyway convergent as $q \geq \gamma >0$. A convergence problem would occur only when $n=2$ in the limit case $\gamma=0$. 
We estimate $I_2$ using also Young inequality as follows
\begin{eqnarray}
I_2&
=&cL\mint_{B_{1}}(\lambda^{2}+\snr{z}^{2}+\delta^{2}\snr{y }^{2}+2\delta z\cdot y )^{\frac{\gamma-2}{2}} \phi (y)\d y \nonumber \\
&\le &cL(\lambda^{2}+\snr{z}^{2}-\snr{z}^{2}/2-8\delta^{2})^{\frac{\gamma-2}{2}}\nonumber \\
& 
=& cL(\lambda^{2}/2+\snr{z}^{2}/2+\lambda^{2}/2-8\delta^{2})^{\frac{\gamma-2}{2}}\nonumber \\
&\leq & cL(\lambda^{2}+\snr{z}^{2})^{\frac{\gamma-2}{2}}
\le cL(\lambda_{\delta}^{2}+\snr{z}^{2})^{\frac{\gamma-2}{2}}\;.\label{sti2}
\end{eqnarray}
We have used that that $\lambda^2/2 - 8 \delta^2>0$. 
As for $I_1$, if $q <2$, we estimate as in the previous display with $q $ instead of $y$, getting that 
$I_1\leq c(\lambda_{\delta}^{2}+\snr{z}^{2})^{\frac{q-2}{2}}$. Otherwise, if $q\geq 2$, this last estimate is trivial. Summarizing, estimate \rif{assGdelta}$_3$ follows when $\lambda>\sqrt{16}\delta$ too. 
Finally we consider the case when $\snr{z}>\sqrt{32}\delta$. We estimate exactly as \rif{sti1} and \rif{sti2}, and we have
$$
I_2
\le cL(\lambda^{2}+\snr{z}^{2}/4+\snr{z}^{2}/4-8\delta^{2})^{\frac{\gamma-2}{2}}
\leq c(\lambda^{2}+\snr{z}^{2})^{\frac{\gamma-2}{2}}
\le cL(\lambda_{\delta}^{2}+\snr{z}^{2})^{\frac{\gamma-2}{2}}\;.
$$
Again, $I_1$ can be estimated in the same way if $q < 2$, otherwise the estimation of $I_1$ becomes trivial. This means that \rif{assGdelta}$_3$ has been completely proved in the case $0 < \gamma < 2$. Concerning \rif{assGdelta}$_2$, we have
\begin{flalign*}
\partial^2 G_{\delta}(z)\, \xi \cdot \xi =&\,  \mint_{B_{1}}\partial^2 G_{\delta}(z+\delta y)\phi(y)   \dy  \, \xi\cdot\xi\nonumber \\
\ge&\, \nu\mint_{B_{1}}(\lambda^{2}+\snr{z+\delta y}^{2})^{\frac{\gamma-2}{2}} \phi(y) \dy\, \snr{\xi}^{2}   \ge \frac{ \nu}{c(\gamma)}(\lambda^{2}_{\delta}+\snr{z}^{2})^{\frac{\gamma-2}{2}}\snr{\xi}^{2}\;,
\end{flalign*}
where we used that $\snr{z+\delta y}\le \snr{z}+\delta$, since $\snr{y }\le 1$. This concludes the proof of \rif{assGdelta} in the case $\gamma <2$.\\
\emph{Case 2: $\gamma\geq 2$.} The upper bound in \rif{assGdelta}$_3$ is trivial. As for \rif{assGdelta}$_2$, recalling the last property in \rif{mollificatori}, we have 
\begin{flalign*}
\partial^2 G_{\delta}(z)\, \xi \cdot \xi 
\ge &\,\frac{\nu}{c(n)} \int_{B_{1}\cap \{y \in B_{1}\colon z\cdot y \ge 0\}}(\lambda^{2}+\snr{z}^{2}+\delta^{2}\snr{y }^{2})^{\frac{\gamma-2}{2}} \phi(y) \dy\, \snr{\xi}^{2}   \nonumber \\
\ge &\, \frac{\nu}{c(n)} \int_{(B_{3/4}\setminus B_{1/2})\cap \{y \in B_{1}\colon z\cdot y \ge 0\}}(\lambda^{2}+\snr{z}^{2}+\delta^{2}\snr{y }^{2})^{\frac{\gamma-2}{2}}\phi(y)  \dy\, \snr{\xi}^{2}   \nonumber \\
\ge &\, \frac{\nu}{c}\left(\int_{B_{3/4}\setminus B_{1/2}}\phi(y) \dy\right) \, (\lambda^{2}+\snr{z}^{2}+\delta^{2}/4)^{\frac{\gamma-2}{2}}\snr{\xi}^{2}\ge \frac{\nu}{c}(\lambda_{\delta}^{2}+\snr{z}^{2})^{\frac{\gamma-2}{2}}\snr{\xi}^{2},
\end{flalign*}
and the proof of \rif{assGdelta} is complete. 
\section{Proof of Theorem \ref{sample1}}\label{primot} We derive Theorem \ref{sample1} as a corollary of Theorem \ref{main3}. For this we check that the assumptions of this last theorem are satisfied; we can assume that $a(\cdot)\in W^{1,r}(\Omega)$ as Theorem \ref{sample1} is local. The integrand $F(x, z)= |z|\log(1+|z|)+a(x)(1+|z|^2)^{q/2}$ satisfies \rif{assF2} with $\lambda=\mu=1$, $\bar{F}(t)=t\log(1+t)$. The bound on $q/(2-\mu)$ in \rif{pq} becomes exactly the assumed one in \rif{pqsample} and it is therefore satisfied too. 
It remains to prove that $\mathcal L^q(u, B)=0$ holds for every ball $B \Subset \Omega$. This can be easily seen by modifying the arguments of \cite[Lemma 13]{sharp} or \cite[Section 13]{BCM3}, that we briefly recall here. Notice that \rif{pqsample} implies that $a(\cdot)\in C^{0, \alpha}$ for $\alpha = 1-n/r$ with $[a]_{0, \alpha }\lesssim \|Da\|_{L^{r}}$ and the bound in \rif{pqsample} reads now
\eqn{boundy}
$$
q < 1+ \alpha/n\;.
$$For $\eps \in (0, 1/2)$ with $\eps \leq {\rm  dist}(B, \partial \Omega)/4$, we take the mollified functions  
$u_\eps :=u*\phi_\eps $ (see \rif{mollificatori}). For every $x \in B$ we define 
$\ai(B_{2\eps}(x)):=\inf \{a(y) \colon  y \in  B_{2\eps}(x)\}$ and $F_\eps(x,z):= |z|^p+ a_{{\rm i}}(B_{2\eps}(x))|z|^q\;.$ It trivially follows that  
$
|Du_\eps(x)| \lesssim  \eps^{-n}$. 
Using this and \rif{boundy} we have
\begin{eqnarray}
\nonumber
F(x, Du_\eps(x)) &\lesssim &[a(x)-a_{{\rm i}}(B_{2\eps}(x))]|Du_\eps(x)|^{q} + F_\eps(x,Du_\eps(x)) +1\\
&\lesssim &[a]_{0,\alpha}\eps^{\alpha}|Du_\eps(x)|^{q-1}|Du_\eps(x)| + F_\eps(x,Du_\eps(x))+1\notag \\
&\lesssim &\eps^{\alpha+n(1-q)}|Du_\eps(x)| + F_\eps(x,Du_\eps(x))+1\notag \\
&\lesssim&  |Du_\eps(x)| + F_\eps(x,Du_\eps(x) )+1\notag \\
&\lesssim& |Du_\eps(x)|\log(1+|Du_\eps(x)|)+F_\eps(x,Du_\eps(x) )+1\notag \\
&\lesssim&
F_\eps(x,Du_\eps(x) )+1\label{prima}
\end{eqnarray}
for every $x\in  B$. 
All the constants involved in the symbol $\lesssim$ above are independent of $\eps$. 
On the other hand, the very definition of $F_\eps(\cdot)$ and Jensen inequality yield 
$$
\nonumber F_\eps(x,Du_\eps(x)) \leq   \int_{B_{\eps}(x)}  F_\eps\left(x, Du(y)\right) \phi_\eps(x-y)\, dy\leq [F(\cdot,Du(\cdot))*\phi_\eps](x)\,.
$$
This last inequality with \rif{prima} gives $F(x, Du_\eps(x))\lesssim [F(\cdot,Du(\cdot))*\phi_\eps](x)+1 $ for every $x\in B$. This implies, by Lebesgue dominated convergence, 
that $F(x, Du_\eps)\to F(x, Du)$ in $L^{1}(B)$ that is, approximation in energy takes place so that $\mathcal {L}(u, B)=0$ follows. 

\section{Non-uniform ellipticity via uniform ellipticity}\label{moser2}
In this final section we give a streamlined version of the a priori estimates technique employed for Theorem \ref{main3} to show how, in a sense, that method allows to reduce the analysis of non-uniformly elliptic functionals to the analysis of uniformly elliptic ones. The key but somehow subtle 
point is a combination of the peculiar dependence on the constants appearing in the standard Moser's iteration (see Lemma \ref{lamoser} below), and in the reverse H\"older inequalities for uniformly elliptic equations (see \rif{revvi1} below). This incorporates and quantifies all the non-uniform ellipticity information, instantaneously leading to sharp a priori estimates. To demonstrate the approach, we consider a 
simplified but yet significant problem; i.e., we take functionals with $(p,q)$-growth as in \rif{genF}, with Lipschitz dependence on $x$, and in the non-degenerate case $\lambda=1$. This is for instance the setting of \cite{M1}, see also \cite{ELM2, sharp}. More in general, non-polynomial growth settings, can be considered too. Specifically, we consider an integrand 
$F \colon \Omega \times \er^n\to [0, \infty)$, which is assumed to be locally $C^2$-regular with respect to the gradient variable. It satisfies 
\eqn{assF2fine}
$$
\left\{
\begin{array}{c}
\nu(1+\snr{z}^{2})^{\frac{p}{2}}\le F(x,z)\le L(1+\snr{z}^{2})^{\frac{q}{2}}\\ [6 pt]
\nu(1+\snr{z}^{2})^{\frac{p-2}{2}}\snr{\xi}^{2}\le \partial_{zz}F(x,z)\, \xi\cdot\xi\\ [6 pt]
\displaystyle
\snr{\partial_{zz}F(x,z)}+\frac{ \snr{\partial_{xz}F(x,z)}}{(1+|z|^2)^{1/2}}\le L(1+\snr{z}^{2})^{\frac{q-2}{2}} \;,
 \end{array}\right.
 $$
with the same notation of \rif{assF2} and for every $x \in \Omega$. Under such assumptions, local minimizers are locally Lipschitz regular provided the condition
\eqn{boundx}
$$
\frac{q}{p}< 1 +\frac1n
$$
is in force together with $\mathcal L^q\equiv 0$. Condition \rif{boundx} is sharp as shown in \cite{sharp, FMM}. The bound in \rif{boundx} corresponds to \rif{pq} when $r \to \infty$ and $2-\mu=p$. In fact \rif{assF2fine} are a particular case of \rif{assF2} when considering $h=1$, $\lambda=1$, $2-\mu=p$ and $r=\infty$ and the Lipschitz continuity of minima follows from Theorem \ref{main3};  see Remark \ref{esp1}. The outcome is that the a priori bound
\eqn{toylip0}
$$
\|Du\|_{L^{\infty}(B_{R/2})} \leq cR^{-\kappa_1}\left[\mathcal{F}(u,B_{R})\right]^{\kappa_2}\,,\qquad B_R \Subset \Omega\,, \quad R\leq1\;,
$$
holds provided $\mathcal L^q(u, B_R)=0$, with $c\equiv  c (n,\nu,L,p,q)$, $\kappa_1, 
\kappa_2\equiv  \kappa_1, 
\kappa_2 (n,p,q)$ (see \rif{toylip} below for $\kappa_1, 
\kappa_2$ when  $n>2$).
\subsection{Review of the standard case $p=q$}\label{stansec} In the standard case, i.e., when \rif{assF2fine} hold with $p=q$, the proof of \rif{toylip0} goes as follows. 
With $H(Du):= 1+|Du|^2$, one proves the reverse type inequality 
\eqn{revvi1}
$$
\left(\int_{B_{\varrho_1}}[H(Du)]^{\left(\gamma+\frac{p}{2}\right)\frac{2^{*}}{2}} \dx \right)^{\frac{2}{2^{*}}}
 \le \frac{\tilde c L^2\tilde R (1+\gamma)^2}{(\varrho_2-\varrho_1)^{2}}\int_{B_{\varrho_2}}[H(Du)]^{\gamma+\frac{p}{2}}  \dx\;.
$$
The number $2^{*}>2$ and $\tilde R$ are defined in \rif{immergi} and \rif{tilder}, respectively.  
Inequality \rif{revvi1} holds for a fixed constant $\tilde c \equiv \tilde c (n,\nu,p)$, for every choice of exponent $\gamma \geq 0$ and concentric balls $B_{\varrho_1} \Subset B_{\varrho_2} \Subset B_{R}$; see Remark \ref{trima} below. 
Here we insist on the explicit dependence on $L^2$ in \rif{revvi1} (and on $\tilde R$ when $n=2$). Inequality \rif{revvi1} allows to use the standard Moser iteration, that is Lemma \ref{lamoser} below (with $\chi=2^{*}/2>1$). This yields
\eqn{localp}
$$
\|H(Du)\|_{L^{\infty}(B_{\tau_1})} \leq \frac{c[L^{2}\tilde R]^{\frac{2}{p}\frac{2^*}{2^*-2}}}{(\tau_2-\tau_1)^{\frac{4}{p}\frac{2^*}{2^*-2}}} \|H(Du)\|_{L^{p/2}(B_{\tau_2})}\;,
$$
where $c \equiv c (n,\nu,p)$ and $B_{\tau_2}\Subset B_{\tau_2}\subset B_R$ are arbitrary concentric balls. Recalling that $2^*/(2^*-2)=n/2$ for $n>2$ and the definition of $\tilde R$, we conclude, for $n\geq 2$, with 
\eqn{localp2}
$$
\|H(Du)\|_{L^{\infty}(B_{R/2})} \leq \frac{c}{R^{2n/p}} \|H(Du)\|_{L^{p/2}(B_{R})
}= c\left(\mint_{B_R}[H(Du)]^{p/2} \dx\right)^{2/p}\;,
$$
which is the usual $L^\infty$-$L^p$-estimate for $p$-harmonic functions and that is equivalent to \rif{toylip0}. 

\begin{remark}\label{trima} \emph{Estimate \rif{revvi1} can be obtained by simply checking the dependence on the constants in the standard proof of the local Lipschitz estimate for $p$-Laplacean type equations. As a matter of fact, \rif{revvi1} follows from \rif{0} taking $2-\mu=p=q$, $r=\infty$ (that means $m=1$) and $h(\cdot)\equiv 1$ (and choosing $\eta$ in the obvious way); see Remark \ref{esp1}. Indeed, with such a choice of the parameters, the proof in Section \ref{step2} becomes the usual proof for the $p$-Laplacean case.} 
\end{remark}
\subsection{Reducing the non-standard case to the standard case} We are going to discuss only the aspects concerning a priori estimates. These have to be anyway embedded in the approximation scheme of Sections \ref{apri} and \ref{chiudi}. Therefore we simply denote $u_{j, \delta}\equiv u$ and $H_{\delta}(\cdot)\equiv H(\cdot)$ with the notation of Sections \ref{apri}-\ref{step2}, and show how to derive \rif{toylip0} directly from the material in the preceding Section \ref{stansec}. Observe that, given a concentric ball $B_{\tau_2}\Subset B_R$, as the problem is local, we can tautologically replace assumption \rif{assF2fine}$_3$ by
\eqn{revvi4} 
$$
\snr{\partial_{zz}F(x,z)}+\frac{ \snr{\partial_{xz}F(x,z)}}{(1+|z|^2)^{1/2}}\le L\|H(Du)\|_{L^{\infty}(B_{\tau_2})}^{\frac{q-p}{2}}(1+\snr{z}^{2})^{\frac{p-2}{2}} \,, \quad \mbox{on $B_{\tau_2}$}\;. 
$$
This, together with \rif{assF2fine}$_2$, formally sets back the functional in the realm of those with standard $p$-growth treated in the Section \ref{stansec}. In fact, these are the only assumptions used to prove a priori estimates as \rif{revvi1}. Therefore estimate \rif{localp} applies and reads
\eqn{revvi30}
$$
\|H(Du)\|_{L^{\infty}(B_{\tau_1})} \leq \frac{c\left[L^2\tilde R \|H(Du)\|_{L^{\infty}(B_{\tau_2})}^{q-p}\right]^{\frac{2}{p}\frac{2^*}{2^*-2}}}{(\tau_2-\tau_1)^{\frac{4}{p}\frac{2^*}{2^*-2}}} \|H(Du)\|_{L^{p/2}(B_{\tau_2})}\;.
$$
Recalling \rif{immergi}, we have that for $n>2$ it is 
$2(q-p)2^*/[p(2^*-2)] = (q-p)n/p$, 
while \rif{boundx} implies 
$(q-p)n/p< 1$. Young inequality then yields
\eqn{revvi3}
$$
\|H(Du)\|_{L^{\infty}(B_{\tau_1})} \leq \frac 12 \|H(Du)\|_{L^{\infty}(B_{\tau_2})}
+ \frac{c\|H(Du)\|_{L^{p/2}(B_{\tau_2})}^{\frac{p}{p-n(q-p)}}}{(\tau_2-\tau_1)^{\frac{2n}{p-n(q-p)}}} \;.
$$
Using Lemma \ref{l2} we come to the final local Lipschitz estimate
\eqn{toylip}
$$
\|H(Du)\|_{L^{\infty}(B_{R/2})} \leq cR^{-\frac{2n}{p-n(q-p)}}\|H(Du)\|_{L^{p/2}(B_{R})}^{\frac{p}{p-n(q-p)}}\;,
$$
with $c\equiv c (n,\nu, L, p,q)$, that indeed coincides with \rif{localp} when $p=q$. 
Estimate \rif{toylip} eventually implies \rif{toylip0} via the approximation argument of Section \ref{chiudi}. In the case $n=2$, the bound in \rif{boundx} is $q/p<3/2$. On the other hand, notice that requiring $2(q-p)2^*/[p(2^*-2)]<1$ means to require that $q/p<3/2- 
1/2^*$ that  can be therefore satisfied by choosing $2^*$ large enough. With this remark we can proceed as after \rif{revvi30}, thereby coming again to \rif{toylip0}. 
\begin{remark}\label{esp2} \emph{When adapting the parameters of Theorem \ref{main3} to catch assumptions \rif{assF2fine} considered here, that is taking $\mu=2-p$, $r=\infty$ (that implies $m=1$, $\tau=0$ and $\sigma=q-p$; see Remark \ref{esp1}), in the case $n>2$ the exponents $\kappa_1, \kk_2$ in \rif{iduekappa} become
$$
\kk_1= \frac{2n}{p+n(q-p)} \quad \mbox{and} \quad
\kk_2= \frac{1}{p-n(q-p)}\;,
$$
respectively. This means that estimate \rif{pre98bis} gives back \rif{toylip} upon taking $\varrho =R/2$. 
}

\end{remark}
\subsection{The classical Moser iteration} For completeness, and to make the arguments of this section more self-contained, we include the proof of the standard Moser's iteration scheme. The only difference with the usual versions scattered in the literature relies in the explicit dependence on the constants. 
\begin{lemma}\label{lamoser} Let $B_R\subset \er^n$ be a ball and let $\tilde H\in L^{p/2}(B_R)$ be a non-negative function such that
\eqn{thenmose0} 
$$
\left(\int_{B_{\varrho_1}}\tilde H^{\left(\gamma+\frac{p}{2}\right)\chi}  \dx\right)^{1/\chi}
 \le \frac{c_1(1+\gamma)^t}{(\varrho_2-\varrho_1)^{s}}\int_{B_{\varrho_2}}\tilde H^{\gamma+\frac{p}{2}}  \dx
$$
holds for every $\gamma \geq 0$, where $c_1, t, \chi, p, s$ are positive constants with $\chi>1$, and where $B_{\tau_1} \Subset B_{\varrho_1} \Subset B_{\varrho_2} \Subset B_{\tau_2}\subset B_R$ are arbitrary concentric balls. Then it holds that
\eqn{thenmose} 
$$
\|\tilde H\|_{L^{\infty}(B_{\tau_1})} \leq c(\chi, s,t)\left[\frac{c_1}{(\tau_2-\tau_1)^s}\right]^{\frac{2}{p}\frac{\chi}{\chi-1}} \|\tilde H\|_{L^{p/2}(B_{\tau_2})}\;.
$$
\end{lemma}
\begin{proof} For integers $k  \geq 1$, we define radii
$\rr_{k }:=\tau_1+(\tau_2-\tau_1)2^{-k +1}$ and exponents $\gamma_{1}:=0$, 
$\gamma_{k +1}:= (\gamma_k +p/2)\chi-p/2$ and $\alpha_{k }:=\gamma_{k}+p/2$, so that \rif{thenmose0} yields
$$
\|\tilde H \|_{L^{\alpha_{k +1}}(B_{\varrho_{k +1}})} \leq   \left[\frac{c_12^{ks }(1+\gamma_k)^t}{(\tau_2-\tau_1)^{s}}\right]^{1/\alpha_{k }}\|\tilde H \|_{L^{\alpha_k }(B_{\varrho_k})}\quad \mbox{and} \quad \alpha_{k +1} = \chi \alpha_{k } = \chi^{k }\alpha_1\;.
$$
Iteration of the above inequalities leads to
\eqn{iteramo}
$$
\|\tilde H \|_{L^{\alpha_{k +1}}(B_{\varrho_{k +1}})} \leq  \prod_{i=1}^{k } \left[\frac{c_12^{i s}(1+\gamma_i)^t}{(\tau_2-\tau_1)^{s}}\right]^{1/\alpha_{i}}\|\tilde H \|_{L^{p/2}(B_{\tau_2})}
$$
again for every $k  \geq 1$. Observing that 
$$
\sum_{i=1}^\infty \frac{1}{\alpha_{i}} = \frac{2}{p}\frac{\chi}{\chi-1}\qquad \mbox{and}\qquad  \prod_{i=1}^{\infty} \left[2^{i s}(1+\gamma_i)^t \right]^{\frac{1}{\alpha_{i}}}\leq \exp\left[\frac{c(s, t, \chi)}{(\chi-1)^2}\right] \;,
$$
letting $k  \to \infty$ in \rif{iteramo} yields \rif{thenmose} and the proof is complete. 
\end{proof}

\end{document}